\documentclass[10pt,a4paper]{article}
\usepackage[english]{babel}
\usepackage[utf8]{inputenc}
\usepackage{amsthm,amsmath}
\usepackage{amsfonts}
\usepackage{amssymb}

\usepackage[T1]{fontenc}
\usepackage{ae,aecompl}

\usepackage{graphicx}
\usepackage{fancyhdr}
\usepackage[left=2cm,right=2cm,top=5cm,bottom=5cm]{geometry}

\usepackage[usenames,dvipsnames]{pstricks}
\usepackage{epsfig}
\usepackage{pst-grad}
\usepackage{pst-plot}
\usepackage[tight,TABTOPCAP]{subfigure}

\newtheorem{ththurston}{Theorem}%

\newtheorem{defgraph}{Definition}[section]
\newtheorem{defisograph}[defgraph]{Definition}
\newtheorem{defplanarmap}[defgraph]{Definition}
\newtheorem{defisomap}[defgraph]{Definition}
\newtheorem{defsubgraph}[defgraph]{Definition}
\newtheorem{defbipmap}[defgraph]{Definition}
\newtheorem{propcnsbip}[defgraph]{Proposition}
\newtheorem{notation}[defgraph]{Notation}

\newtheorem{definbipmap}[defgraph]{Definition}
\newtheorem{defdegree}[defgraph]{Definition}
\newtheorem{propconnconv}[defgraph]{Proposition}
\newtheorem{propdegbwvf}[defgraph]{Property}
\newtheorem{propdegunin}[defgraph]{Property}
\newtheorem{theuler2}[defgraph]{Theorem}

\newtheorem{defSDR}[defgraph]{Definition}
\newtheorem{defmarcond}[defgraph]{Definition}
\newtheorem{thhall}[defgraph]{Theorem}

\newtheorem{thgamcon}{Theorem}[section]
\newtheorem{defrepresentation}[thgamcon]{Definition}
\newtheorem{proprepre}[thgamcon]{Proposition}
\newtheorem{proplabel}[thgamcon]{Proposition}

\newtheorem{defrealiz}[thgamcon]{Definition}
\newtheorem{defbalmap}[thgamcon]{Definition}
\newtheorem{lemconvsubmap}[thgamcon]{Lemma}
\newtheorem{corconv}[thgamcon]{Corollary}

\newtheorem{propwhitedeg}[thgamcon]{Proposition}
\newtheorem{thbalmaps}[thgamcon]{Theorem}
\newtheorem{corbalmaps}[thgamcon]{Corollary}

\newtheorem{thbalmap}[thgamcon]{Theorem}
\newtheorem{corbalmap}[thgamcon]{Corollary}

\newtheorem{propRHcond}{Proposition}[section]
\newtheorem{defassmat}[propRHcond]{Definition}
\newtheorem{defloccond}[propRHcond]{Definition}
\newtheorem{thcond}[propRHcond]{Theorem}
\newtheorem{thexists}[propRHcond]{Theorem}

\begin{document}

\renewcommand{\proofname}{\textbf{Proof}}

\begin{Large}
\begin{center}
\textbf{\large{Realizations of self branched coverings of the $2$-sphere}}
\end{center}
\end{Large}

\begin{center}
\textbf{J. Tomasini}\footnote{partially supported by ANR-13-BS01-0002, project LAMBDA}
\end{center}
\begin{small}
\begin{center}
\textit{Laboratoire Angevin de REcherche en MAthématiques (LAREMA)} - \textit{Université d'Angers}

\today
\end{center}
\end{small}

---------------------------------------------------------------------------------------------------------------------------------------------
\begin{abstract}
For a degree $d$ self branched covering of the $2$-sphere, a notable combinatorial invariant is an integer partition of $2d-2$, consisting of the multiplicities of the critical points. A finer invariant is the so called Hurwitz passport. The realization problem of Hurwitz passports remain largely open till today. In this article, we introduce two different types of finer invariants: a bipartite map and an incident matrix. We then settle completely their realization problem by showing that a map, or a matrix, is realized by a branched covering if and only if it satisfies a certain balanced condition.
A variant of the bipartite map approach was initiated by W. Thurston. Our results shed some new lights to the Hurwitz passport problem.
\end{abstract}

\vspace{3mm}

\begin{small}

\end{small}

---------------------------------------------------------------------------------------------------------------------------------------------

\section{Introduction}

The main topic of this article is the study of branched covers from $\mathbb{S}^2$ to $\mathbb{S}^2$.
Typical examples are meromorphic functions defined on the Riemann sphere. Our objective is to introduce a new set of combinatorial properties of  these branched coverings and solve the related realization problems. 

A map $\pi: \mathbb{S}^2 \rightarrow \mathbb{S}^2$ is called {\bf a branched (or ramified) covering} of degree $d$, if there exists some finite subset $F \subset \mathbb{S}^2$ such that
\begin{itemize}\item the restriction map $\pi : \mathbb{S}^2 \setminus \pi^{-1}(F) \rightarrow \mathbb{S}^2 \setminus F$ is a covering map of degree $d$;
\item for each point $x \in F$, there is a neighborhood $V$ of $x$, and a neighborhood $U$ of  each preimage $y$ of $x$ by $\pi$,  such that the restricted map $\pi|_U : U \rightarrow V$ is equivalent, up to topological change of coordinates,  to the map $z \mapsto z^k$ on the unit disc, for some integer $k>0$. 
\end{itemize}
This integer is uniquely determined for any point $y$ of $\pi^{-1}(F)$ (and more generally for any point of $\mathbb{S}^2$), and it's called the \textbf{ramification number} of $y$. Notice that this number is equal to the number of preimages, close to $y$, of a point close to $x$. In particular, this integer is equal to $1$ if and only if $\pi$ is locally a homeomorphism.

A point whose the ramification number is greater than $1$ is called a \textbf{critical point}, and its image a \textbf{critical value}. Moreover, the sum of the ramification numbers of the preimages of any point in $\mathbb{S}^2$ is constant, equal to the degree of the branched covering. In other words,  the ramification numbers of the preimages of any point $x$ in $\mathbb{S}^2$ form an integer partition of $d$. And this partition is not the trivial partition $(1,\ldots, 1)$ if and only if $x$ is a critical value.

So, let $\pi$ be a branched covering $\pi$ of degree $d$, then we can associate to $\pi$ two combinatorial properties:
\begin{itemize}
\item the (unordered) list $l=(a_1,\ldots,a_m)$ of ramification numbers at the critical points, with $2 \leq a_i \leq d$ for every $i$;
\item the \textbf{branch datum} $\mathcal{D} = [\Pi_1, \ldots , \Pi_n]$, with each $\Pi_i$ a non trivial integer partition of $d$, representing the collection of ramification numbers of the preimages of the $i$-th critical value. 
\end{itemize}

Here, the integers $m$ and $n$ are respectively the number of critical points and the number of critical values of $\pi$. Notice also that the branch datum incorporates the information of all the other properties of $\pi$. For example, consider the application $f: \hat{\mathbb{C}} \rightarrow \hat{\mathbb{C}}$ defined by:
\begin{displaymath}
f(x) = \frac{(x^2-1)^2}{x^3}.
\end{displaymath}
This function defines a branched covering of the sphere of degree $4$ with $5$ critical points ($0$, $\pm 1$, $\pm i$), $4$ critical values ($0$, $\pm 4i$, $\infty$), its list of ramification numbers is $(3,2,2,2,2)$ and its branched datum is $[[3,1],[2,2],[2,1,1],[2,1,1]]$.

Our focus will be on the {\bf realization problem} of these combinatorial properties. It can be expressed as follows:

 \bigskip
 
{\noindent\bf Realization problem} {\em Consider an integer $d\ge 2$. Given a list $l=(a_1, \ldots , a_m)$ of integers, with $2 \leq a_i \leq d$, or a list $\mathcal{D}$ of non-trivial integer partition of $d$, can it be realized by a branched covering of $\mathbb{S}^2$ of degree $d$?}

 \bigskip
 
 This problem, in particular the realization problem of an abstract branch datum, is generally called the {\bf Hurwitz problem}. There exists an important necessary condition, called the \textbf{Riemann-Hurwitz condition} (we will reprove it along the way). 

Let  $\Pi_i = [k_1, \ldots ,k_{n_i}]$ be an integer partition of $d$. We define the \textbf{weight} of  $\Pi_i$, denoted $\nu(\Pi_i)$, by
\begin{displaymath}
\nu(\Pi_i) = \sum_{j=1}^{n_i}{(k_j-1)}.
\end{displaymath}
We define the \textbf{branched weight} (or total weight) of a list  $\mathcal{D}=[\Pi_1, \ldots , \Pi_n]$ of such partitions by
\begin{displaymath}
\nu(\mathcal{D}  = \sum_{i=1}^{n}{\nu(\Pi_i)}.
\end{displaymath} 

{\noindent\textbf{Riemann-Hurwitz condition} {\em If a list  $\mathcal{D}=[\Pi_1, \ldots , \Pi_n]$ of integer partitions of $d$ is the branch datum of a branched covering of degree $d$, then
\begin{equation}
\label{HRcond2}
\nu(\mathcal{D}) = 2d-2.
\end{equation}
}

We can also rewrite this condition for the list of ramification numbers:

{\noindent\textbf{Riemann-Hurwitz condition} {\em If a list of integers $l=(a_1,\ldots,a_m)$ is the list of ramification numbers at the critical points of a branched covering of degree $d$, then $2\le a_i\le d$ for every $i$ and $\sum_i (a_i-1)=2d-2$.
}

\vspace{3mm}
Notice that the numbers $m$ of critical points, and $n$ of critical values of a branched covering of degree $d\ge 2$ satisfy $2\le n\le m \le 2d-2$.

A list that satisfies the Riemann-Hurwitz condition will be called a \textbf{ramification distribution of degree $d$}, and a list $\mathcal D$ of integer partitions of $d$ that satisfies the Riemann-Hurwitz condition is generally called a \textbf{passport} in the literature. The classical Hurwitz problem can then be expressed as follows: which passport is realized? 

There is a huge literature on this problem. Many partial results have been obtained, using  either the initial approach of Hurwitz, see for example \cite{EKS} or \cite{Ger}), or using a more geometric approach, see for example \cite{Bar} or \cite{Zheng}. Nevertheless,  this problem remains largely open till today.

Following an initial work of W. Thurston, we will introduce a new set of combinatorial properties, namely a bipartite map and its incident matrix, all incorporating the information about the ramification numbers at the critical points. We then solve completely their realization problem.

\subsection*{Thurston's geometrical approach}
\label{secthu}

During the autumn 2010, W. Thurston conducted a group discussion through e-mails on what he called the shapes of rational maps.  Along the way, he established a beautiful result, which becomes the starting point of our present study. Here is an account of his statement. For a presentation of the initial proof of W. Thurston, 
the reader can refer to the manuscript of Sarah Koch and Tan Lei \cite{Tan}.

\vspace{2mm}

Consider $\pi : \mathbb{S}^2 \rightarrow \mathbb{S}^2$ a {\bf generic branched covering} of the sphere of degree $d$, i.e. $\pi$ has exactly $2d-2$ distinct critical values. Then, pick a Jordan curve $\Sigma$ running through the critical values. For example, for a rational map with all critical values real, one may pick  the real axis union the point at infinity as $\Sigma$. We will consider $\Sigma$ as a planar map (i.e. a connected graph drawn on the sphere) with $2d-2$ vertices, each of valence $2$.

Now, pulling back $\Sigma$ by $\pi$, we get a new planar map  $\pi^{-1}(\Sigma)$. Forgetting the $2$-valence vertices of $\pi^{-1}(\Sigma)$, we get a further planar map $\Gamma$ with $2d-2$ vertices (the critical points of $\pi$), all of valence $4$. We call $\Gamma$ the underlying $4$-valent map of $\pi^{-1}(\Sigma)$. This is the combinatorial object we will be focused on. 
Thurston solved  its realization problem as follows:

\begin{ththurston}
\label{ththurston}
(Thurston, \cite{Tan}) A $4$-valent planar map $\Gamma$ is realized by a generic branched covering $\pi$ as   the underlying map of  $\pi^{-1}(\Sigma)$ for some choice of  $\Sigma$, if and only if it satisfies the following \textbf{balanced} condition:
\begin{enumerate}
\item[1.] \textbf{(global balance)} In an alternative coloring of the complementary faces of $\Gamma$, there are equal numbers of white and black faces.
\item[2.] \textbf{(local balance)} For any oriented simple closed curve drawn on the map that keeps black faces on the left and white faces on the right (except at the corners), there are strictly more black faces than white faces on the left side.
\end{enumerate}
\end{ththurston}

The key idea to prove this theorem is to translate the realization problem into finding a pattern of dots so that each face is incident to exactly $2d-2$ vertices, and then into a matching problem on the map. More exactly, consider a $4$-valent planar map $\Gamma$ and put $2d-2$ men in each black face, and $2d-2$ women in each white face. Then, for each face, remove one person per vertex incident to this face. Each remaining person is trying to find a partner from one of the neighboring faces.

Thurston proves that the usual marriage criterion (see Theorem \ref{thhall}) can be reduced to the global and local balance conditions in Theorem \ref{ththurston}. An interested reader may compare the geometrical point of view of Thurston to that of Baranski in \cite{Bar}. 

\subsection*{Our contributions}

We will conduct a self-contained study, similar to that of Thurston, for all branched coverings. In a way,
we recover Thurston's result in the generic case, with a similar idea of proofs, but with a different presentation. 

For our purpose, we find it convenient to pull back a different map than a Jordan curve running through 
the critical values. 

Let $\pi: \mathbb{S}^2 \to \mathbb{S}^2$ be a branched covering of degree $d\ge 2$. 
Color white each of its critical values. Pick a non critical value point and color it black. 
Pick a collection of pairwise disjoint 'legs' connecting the black vertex to each white vertex. We obtain
a bipartite map $T$, i.e\@. a planar map whose edges are connected to a white and a black vertex. 

Now pulling back by $\pi$ the map $T$ and then erasing every  1-valence white vertex together with
its unique incident edge, we obtain a bipartite map $G$, called a {\bf skeleton} of $\pi$. This is the first  combinatorial object we will be focused on. We will establish the following realization results:

{\noindent \bf Theorem A} {\em A planar bipartite map $G$ is realized as a skeleton of a branched covering if and only if it satisfies the following balanced condition:
\begin{itemize}
\item (\textbf{global balance condition}) $G$ has as many black vertices as faces.
\item (\textbf{local balance condition}) Each submap $H$ of $G$ containing at least one black vertex has at least as many black vertices as faces. 
\end{itemize}}

A bipartite map, or more generally a bipartite graph (see definitions below) induces a white-to-black incident matrix $A=(a_{ij})$ as follows: numerate separately the black vertices and the white ones, and set $a_{ij}=k$ if there are $k$ edges connecting the $i$-th white vertex to the $j$-th black vertex. Our next realization result can be stated as follows:

\bigskip

{\noindent \bf Theorem B} {\em Let $d\ge 2$ be an integer. Pick an integer $1\le m\le 2d-2$ and an $m\times d$ matrix $A=(a_{ij})$. Then $A$ is realized as the incident matrix of a skeleton of a degree $d$ branched covering if and only if it satisfies the following conditions:
\begin{itemize}
\item the bipartite graph associated to $A$ is connected and admits an embedding into the plane.
\item (\textbf{global balance condition}) The column vector $(b_i)$ obtained by summing up the columns of $A$ satisfies $\sum_i (b_i-1)= 2d -2$.
\item (\textbf{local balance condition}) For any integer $1\le k< d$, and any choice of $k$ distinct columns of $A$, the column vector $(c_i)$ obtained by summing up the selected columns satisfies $\sum_i (c_i-1)\le 2k -2$.
\end{itemize}}

Notice that these conditions on the matrix $A$ are invariant under permutations of lines and columns of $A$, so do not depend on the numerations of the vertices of its bipartite graph. As an easy consequence, we obtain a solution of one of our initial realization problem:

\bigskip

{\noindent \bf Theorem C} {\em Let $d\ge 2$ be an integer. Pick a list $l=(a_1, \ldots, a_m)$ of integers, with $2 \leq a_i \leq d$ for every $i$. Then $l$ is realized as the list of  critical ramification numbers of a degree $d$ branched covering if and only if it satisfies the following condition:
\begin{itemize}
\item (\textbf{global balance condition}) $\sum_i (a_i-1)=2d-2$.
\end{itemize}}

The necessity comes from the Riemann-Hurwitz condition. It is interesting to notice that there is no local condition for this last realization result. In particular, this result supposes that the construction of a branched datum from a list of integers is hidden in this local condition. For example, consider the list $(3,2,2,2,2)$, then this list satisfies the Riemann-Hurwitz condition for $d=4$, and the realized passports we can construct from this list are $[[3,1],[2,1],[2,1],[2,1],[2,1]]$ or $[[3,1],[2,2],[2,1],[2,1]]$, but the list $[[3,1],[2,2],[2,2]]$ is not a realized passport. 

Here appears another interesting question: consider a ramification distribution of degree $d$, for a given integer $d$. What is the minimal number of partition of $d$ we need to construct a realized passport? In our example, this number is equal to $4$. The problem of minimality is a difficult problem, and will be considered in future project. The interested reader can still find an equivalent problem in (\cite{Tomth},p.101).

\subsection*{Contents}

The structure of this article is based on the following schema. After introducing all the tools necessary for the understanding and the construction of the combinatorics used here, we begin the third section by the study of the branched coverings of $\mathbb{S}^2$ to itself, leading to a classification up to topological equivalence. The starting point of this work is the characterization of Thurston we mentioned earlier. Then, we give a new characterization of the balanced maps that allow us to ``extend'' this notion on the set of graphs. 

We finish this section by proving one of the main result of this article (Theorem A) which is a generalization of the result of Thurston (Theorem \ref{ththurston}) that proves there exists a bijection between what we call a skeleton of a branched covering, and the set of balanced maps. 

In the last section, we come back to the Hurwitz problem mentioned earlier. More exactly, after giving a relation between the global balanced condition on our maps and the Riemann-Hurwitz condition, we establish a matrix interpretation of the balanced conditions defined in the third section (Theorem B), we finally prove that a ramification distribution of degree $d$ always comes from a branched covering of degree $d$ (Theorem C).

\section{Combinatorial tools}

\subsection{Increasing bipartite maps}

The purpose of this subsection is to introduce the main combinatorial objects, and some constructions that will be useful in the following. As there is many manner to introduce planar maps, we begin this subsection by introducing the main definition in order to fix the notion and notation we use in this article. The approach chosen here comes from the books of C. Berge \cite{berge}, W. T. Tutte \cite{tutbook}, P. Flajolet and R. Sedgewick \cite{FS}, or again J. L. Gross and T. W. Tucker \cite{tucker}.

\begin{defgraph}
A graph $G$ is given by two finite sets $\mathcal{E}_G$ and $\mathcal{V}_G$ and an application $\rho_G$ from $\mathcal{E}_G$ to $\mathcal{V}_G \times \mathcal{V}_G$. 
\end{defgraph}

An element of $\mathcal{V}_G$ (resp. $\mathcal{E}_G$) is called a \textbf{vertex} (resp. an \textbf{edge}) of $G$. For each edge $a$ of $\mathcal{E}_G$ is associated a couple of vertices $\rho(a)$ forming the \textbf{tips} of $a$. We say that two vertices of $G$ are \textbf{adjacent} if they are the tips of an edge. \newline
An edge whose tips are the same is called a \textbf{loop} of $G$. Finally, we say that an edge $a$ is \textbf{incident} to a vertex $v$ if $v$ is one of the two tips of $a$. In this case, we say also that $v$ is incident to $a$. Particularly, a loop of tip $v$ is doubly incident to the vertex $v$. \newline
The \textbf{degree} of a vertex $v$ of $G$ is the number of edges of $G$ being incident to $v$, counted with multiplicity. Particularly, if $G$ is a graph with $n$ edges, the sum of the degree of all the vertices of $G$ is equal to $2n$, i.e the degrees of the vertices give a partition of $2n$.

Notice that this definition of graph is a combinatorial definition, but we often consider in the following a graph as a topological object, i.e a set of points (the vertices) link by some line (the edges). 

\begin{defisograph}
Let $G$ and $H$ be two graphs. We say $G$ and $H$ are isomorphic if there exists a bijection $\phi : \mathcal{V}_G \rightarrow \mathcal{V}_H$ and a bijection $\psi : \mathcal{E}_G \rightarrow \mathcal{E}_H$ such that $\rho_H \circ \phi = \psi \circ \rho_G$. In other words, the bijections $\phi$ and $\psi$ preserve the incident relation.
\end{defisograph}

This relation between graphs defines an equivalence relation. In the following, graphs are always considered up to this equivalence relation. Now, we can define the notion of map. \newline
We call \textbf{embedding} of a graph $G$ a repreasentation of $G$ in a surface such that the edges of $G$ only have intersection on the vetices of $G$. Notice that such an embedding is not always possible in a given surface (see "problème des trois ponts" for example). If a graph admits such an embedding on the Riemann sphere, or equivalently on the plane, we say that $G$ is a \textbf{planar graph}. The embedding of a graph $G$ on a surface divise this surface into some connected components called \textbf{faces}. An embedding is called \textbf{cellular} if each face is simply connected.

\begin{defplanarmap}
A \textbf{planar map} is a cellular embedding of a graph $G$ in the Riemann sphere.
\end{defplanarmap}

In particular, a planar map is connected. As for the graph, we can define some incident relation between vertices-faces, vertices-edges and edges-faces. Notice that an edge is incident to exactly two faces (eventually the same) whereas a vertex can be incident to an arbitrary number of faces. 

For convenience, we draw a planar map on the plane and not on the sphere, but this drawing is only a representation of the planar map. In fact, by projecting the sphere on the plane via a stereographic projection, we send a point of the sphere to infinity, so depending of the choice of the projection we obtain different representations of the same planar map. 

\begin{defisomap}
Let $G$ and $H$ be two planar maps. We say that $G$ and $H$ are isomorphic if the vertices, edges and faces of one map can be sent bijectively on the vertices, edges and faces of the second one, preserving the three incident relations.
\end{defisomap}

Once again, this relation define an equivalence relation and so from now planar maps are considered up to this equivalence relation. We can now define the notion of submap.

\begin{defsubgraph}
Let $G$ be a planar map. We say that $H$ is a \textbf{submap} of $G$, denoted by $H \subset G$, if $H$ is itself a planar map and $H$ can be obtained by erasing some edges and vertices of $G$. Moreover, we say that $H$ is a \textbf{full submap} of $G$, denoted by $H < G$, if $H$ is a submap of $G$ and there exists at most one face $f$ of $H$ such that every edge of $G$ not in $H$ is contained in $f$.
\end{defsubgraph}

By convention, $G$ itself is a (full) submap of $G$. Note that if $H$ is a submap of $G$, then the set of vertices, resp. edges, of $H$ is a subset of those of $G$, but a face of $H$ either coincides with a face of $G$ or contains several faces of $G$. The reader can see examples of full submaps in Figure \ref{figcsubmap}.

This terminology of full submap is motivated by the notion of planar full continua namely compact sets that are connected with a connected complement in the $2$-sphere.

\begin{defbipmap}
\label{defbipmap}
Let $G$ be a planar map with at least two vertices. We say that $G$ is \textbf{(black and white) bipartite} if each vertex of $G$ is colored in either white or black and each edge of $G$ has a black and a white tip (in particular a bipartite map has no loop).

A \textbf{tree} is a bipartite map with a unique face.

A \textbf{star-like} map is a bipartite tree with a unique black vertex.
\end{defbipmap}

\begin{figure}[ht]
\begin{center}
\psscalebox{0.8 0.8} 
{
\begin{pspicture}(0,-1.8996261)(5.199252,1.8996261)
\psdots[linecolor=black, dotsize=0.2](0.09855774,0.20106842)
\psdots[linecolor=black, dotsize=0.2](3.4985578,1.8010684)
\psdots[linecolor=black, dotsize=0.2](3.4985578,0.20106842)
\psdots[linecolor=black, dotsize=0.2](4.2985578,-1.3989316)
\psline[linecolor=black, linewidth=0.04](0.09855774,0.20106842)(1.6985577,0.8010684)
\psline[linecolor=black, linewidth=0.04](1.6985577,0.8010684)(3.4985578,1.8010684)
\psline[linecolor=black, linewidth=0.04](3.4985578,1.8010684)(5.098558,0.60106844)
\psline[linecolor=black, linewidth=0.04](5.098558,0.60106844)(3.4985578,0.20106842)
\psline[linecolor=black, linewidth=0.04](3.4985578,0.20106842)(1.6985577,0.8010684)
\psline[linecolor=black, linewidth=0.04](0.09855774,0.20106842)(1.2985578,-1.7989316)
\psline[linecolor=black, linewidth=0.04](1.2985578,-1.7989316)(4.2985578,-1.3989316)
\psline[linecolor=black, linewidth=0.04](4.2985578,-1.3989316)(3.2985578,-0.7989316)
\psline[linecolor=black, linewidth=0.04](3.2985578,-0.7989316)(3.4985578,0.20106842)
\psline[linecolor=black, linewidth=0.04](5.098558,0.60106844)(4.2985578,-1.3989316)
\psdots[linecolor=black, dotstyle=o, dotsize=0.2](3.2985578,-0.7989316)
\psdots[linecolor=black, dotstyle=o, dotsize=0.2](5.098558,0.60106844)
\psdots[linecolor=black, dotstyle=o, dotsize=0.2](1.6985577,0.8010684)
\psdots[linecolor=black, dotstyle=o, dotsize=0.2](1.2985578,-1.7989316)
\end{pspicture}
}
\end{center}
\caption{Example of a bipartite map $G$}
\label{figbipmap}
\end{figure}
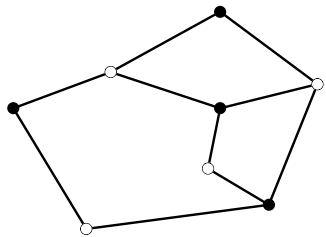

Figure \ref{figbipmap} shows an example of bipartite map. Moreover, there exists a sufficient and necessary condition such that a given planar map is bipartite:

\begin{propcnsbip}
A planar map $G$ can be made bipartite if and only if each face of $G$ is incident to an even number of edges, counted with multiplicity.
\end{propcnsbip}

This condition is clearly necessary. To prove the condition is sufficient, we can use an idea similar to that used later to prove Proposition \ref{proplabel} (details are left to the reader).

\vspace{2mm}

Bipartite maps will be used in the following to describe branched covering of the sphere, where black and white vertices play very different roles, that's why we need always to distinguish these two kind of vertices. In the following, we always use the following notations:

\begin{notation}
Let $G$ be a bipartite map. We denote by: 

\begin{itemize}
\item $\mathcal{V}_G$ the set of black vertices of $G$, and by $V_G$ the cardinal of $\mathcal{V}_G$.
\item $\mathcal{W}_G$ the set of white vertices of $G$, and by $W_G$ the cardinal of $\mathcal{W}_G$.
\item $\mathcal{E}_G$ the set of edges of $G$, and by $E_G$ the cardinal of $\mathcal{E}_G$.
\item $\mathcal{F}_G$ the set of faces of $G$, and by $F_G$ the cardinal of $\mathcal{F}_G$.
\end{itemize}
\end{notation}

To conclude these definitions, let's introduce the notion of increasing bipartite maps.

\begin{definbipmap}
We say that $G$ is a \textbf{increasing bipartite map} if $G$ is a bipartite map with a numbering on each white vertex such that for each black vertex of $G$, the labels of its adjacent (white) vertices appear in an increasing order counterclockwise.
\end{definbipmap}

Notice that an increasing bipartite map can not have multiple edges, i.e. two vertices are connected by at most one edge. There are some examples of increasing bipartite maps in Figure \ref{figinbipmap} and \ref{figrepresent}.

Now that we have introduced the general context in which we will work, we can focus a little more on notions and constructions specific to these maps. First, we define the notion of degree of a face.

\begin{defdegree}
Let $G$ be a bipartite map, and let $f$ be a face of $G$. The \textbf{degree} of $f$, denoted by $\deg_G(f)$, or simply $\deg(f)$, is equal to the number of white vertices on the boundary of $f$, counted with multiplicity.
\end{defdegree}

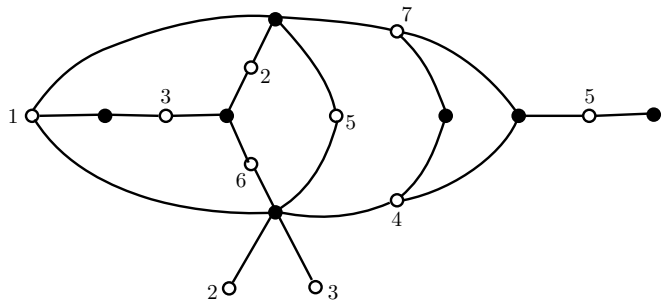
\begin{figure}[ht]
\begin{center}
\psscalebox{0.8 0.8} 
{
\begin{pspicture}(0,-2.4342053)(10.738269,2.4342053)
\pscircle[linecolor=black, linewidth=0.04, fillstyle=solid,fillcolor=black, dimen=outer](1.6,0.6760193){0.12}
\pscircle[linecolor=black, linewidth=0.04, fillstyle=solid,fillcolor=black, dimen=outer](3.6,0.6760193){0.12}
\pscircle[linecolor=black, linewidth=0.04, fillstyle=solid,fillcolor=black, dimen=outer](4.4,2.2760193){0.12}
\pscircle[linecolor=black, linewidth=0.04, fillstyle=solid,fillcolor=black, dimen=outer](4.4,-0.9239807){0.12}
\pscircle[linecolor=black, linewidth=0.04, fillstyle=solid,fillcolor=black, dimen=outer](7.2,0.6760193){0.12}
\pscircle[linecolor=black, linewidth=0.04, fillstyle=solid,fillcolor=black, dimen=outer](8.4,0.6760193){0.12}
\pscircle[linecolor=black, linewidth=0.04, dimen=outer](2.6,0.6760193){0.12}
\pscircle[linecolor=black, linewidth=0.04, dimen=outer](4.0,1.4760193){0.12}
\pscircle[linecolor=black, linewidth=0.04, dimen=outer](4.0,-0.123980716){0.12}
\pscircle[linecolor=black, linewidth=0.04, dimen=outer](5.4,0.6760193){0.12}
\pscircle[linecolor=black, linewidth=0.04, dimen=outer](6.4,2.0760193){0.12}
\pscircle[linecolor=black, linewidth=0.04, dimen=outer](6.4,-0.7239807){0.12}
\pscircle[linecolor=black, linewidth=0.04, dimen=outer](0.4,0.6760193){0.12}
\psline[linecolor=black, linewidth=0.04](1.6,0.6960193)(0.52,0.6760193)
\psline[linecolor=black, linewidth=0.04](1.58,0.6960193)(2.48,0.6760193)
\psline[linecolor=black, linewidth=0.04](3.58,0.6960193)(2.7,0.6760193)
\psline[linecolor=black, linewidth=0.04](3.62,0.7360193)(3.94,1.3960193)
\psline[linecolor=black, linewidth=0.04](4.38,2.2760193)(4.04,1.5760193)
\psline[linecolor=black, linewidth=0.04](4.4,-0.8839807)(4.06,-0.20398071)
\psline[linecolor=black, linewidth=0.04](3.62,0.6760193)(3.96,-0.0980713)
\psbezier[linecolor=black, linewidth=0.04](4.42,2.3160193)(3.4161134,2.41401)(2.4249623,2.0960782)(2.0,1.9360193)(1.5750377,1.7759606)(1.0211877,1.6472968)(0.42,0.7760193)
\psbezier[linecolor=black, linewidth=0.04](4.42,-0.9239807)(3.618616,-0.983528)(2.7621162,-0.8389971)(2.2633498,-0.6768202)(1.7645833,-0.5146433)(0.93212664,-0.17916456)(0.44,0.6160193)
\psbezier[linecolor=black, linewidth=0.04](4.4,2.2960193)(4.88,1.7760193)(5.28,1.2760193)(5.38,0.7760193)
\psbezier[linecolor=black, linewidth=0.04](4.4,-0.90398073)(4.86,-0.64398074)(5.22,-0.18398072)(5.42,0.59601927)
\psbezier[linecolor=black, linewidth=0.04](4.38,-0.90398073)(5.0,-1.0439807)(5.64,-1.0439807)(6.28,-0.7639807)
\psbezier[linecolor=black, linewidth=0.04](4.4,2.3160193)(5.0,2.2760193)(5.74,2.2360194)(6.32,2.0960193)
\psbezier[linecolor=black, linewidth=0.04](7.22,0.6960193)(7.0,1.2560192)(6.9,1.5760193)(6.44,2.0160193)
\psbezier[linecolor=black, linewidth=0.04](7.2,0.6560193)(7.04,0.17601928)(6.9,-0.28398073)(6.46,-0.6639807)
\psbezier[linecolor=black, linewidth=0.04](8.38,0.7160193)(7.9,1.4160193)(7.16,1.9360193)(6.5,2.0560193)
\psbezier[linecolor=black, linewidth=0.04](8.4,0.6760193)(8.22,0.116019286)(7.4,-0.5439807)(6.5,-0.7239807)
\rput[bl](0.0,0.5560193){$1$}
\rput(2.6,0.9960193){$3$}
\rput(4.24,1.3360193){$2$}
\rput(3.84,-0.38398072){$6$}
\rput(5.64,0.5760193){$5$}
\rput(6.4,-1.0439807){$4$}
\rput(6.56,2.3560193){$7$}
\pscircle[linecolor=black, linewidth=0.04, dimen=outer](3.64,-2.1839807){0.12}
\pscircle[linecolor=black, linewidth=0.04, dimen=outer](5.06,-2.1639807){0.12}
\psline[linecolor=black, linewidth=0.04](4.38,-0.90398073)(3.7,-2.0839808)
\psline[linecolor=black, linewidth=0.04](4.42,-0.9239807)(5.0,-2.0439806)
\rput[bl](3.28,-2.3639808){$2$}
\rput[bl](5.26,-2.3639808){$3$}
\pscircle[linecolor=black, linewidth=0.04, dimen=outer](9.56,0.6760193){0.12}
\pscircle[linecolor=black, linewidth=0.04, fillstyle=solid,fillcolor=black, dimen=outer](10.62,0.6960193){0.12}
\psline[linecolor=black, linewidth=0.04](8.42,0.6760193)(9.46,0.6760193)
\psline[linecolor=black, linewidth=0.04](10.62,0.7160193)(9.66,0.6760193)
\rput[bl](9.48,0.8960193){$5$}
\end{pspicture}
}
\end{center}
\caption{Example of an increasing bipartite map.}
\label{figinbipmap}
\end{figure}

For example, in Figure \ref{figinbipmap} the unbounded face has degree $7$ (the vertex labeled $5$ counts twice, whereas the vertices labeled $2$ and $3$ count only once). Notice that this definition of degree is not the usual one, but is more convenient in this context, the role of the white vertices being predominant in the following.

\begin{figure}[ht]
\begin{center}
\psscalebox{0.8 0.8} 
{
\begin{pspicture}(0,-2.8006945)(7.397115,2.8006945)
\psdots[linecolor=black, dotsize=0.2](1.0985577,2.5)
\psdots[linecolor=black, dotsize=0.2](4.8985577,2.5)
\psdots[linecolor=black, dotsize=0.2](3.0985577,0.30000007)
\psdots[linecolor=black, dotsize=0.2](4.0985575,-1.0999999)
\psdots[linecolor=black, dotsize=0.2](7.2985578,-1.3)
\psdots[linecolor=black, dotsize=0.2](2.2985578,-2.6999998)
\psdots[linecolor=black, dotsize=0.2](0.098557666,-0.8999999)
\psline[linecolor=black, linewidth=0.04](1.0985577,2.5)(0.29855767,1.1)
\psline[linecolor=black, linewidth=0.04](0.29855767,1.1)(0.098557666,-0.8999999)
\psline[linecolor=black, linewidth=0.04, linestyle=dashed, dash=0.17638889cm 0.10583334cm](0.098557666,-0.8999999)(0.6985577,-2.3)
\psline[linecolor=black, linewidth=0.04, linestyle=dashed, dash=0.17638889cm 0.10583334cm](0.6985577,-2.3)(2.2985578,-2.6999998)
\psline[linecolor=black, linewidth=0.04, linestyle=dashed, dash=0.17638889cm 0.10583334cm](2.2985578,-2.6999998)(5.2985578,-2.6999998)
\psline[linecolor=black, linewidth=0.04, linestyle=dashed, dash=0.17638889cm 0.10583334cm](5.2985578,-2.6999998)(7.2985578,-1.3)
\psline[linecolor=black, linewidth=0.04, linestyle=dashed, dash=0.17638889cm 0.10583334cm](7.2985578,-1.3)(6.698558,1.3000001)
\psline[linecolor=black, linewidth=0.04, linestyle=dashed, dash=0.17638889cm 0.10583334cm](6.698558,1.3000001)(4.8985577,2.5)
\psline[linecolor=black, linewidth=0.04](4.8985577,2.5)(3.0985577,2.7)
\psline[linecolor=black, linewidth=0.04](3.0985577,2.7)(1.0985577,2.5)
\psline[linecolor=black, linewidth=0.04, linestyle=dashed, dash=0.17638889cm 0.10583334cm](1.0985577,2.5)(2.6985576,1.7)
\psline[linecolor=black, linewidth=0.04, linestyle=dashed, dash=0.17638889cm 0.10583334cm](2.6985576,1.7)(3.0985577,0.30000007)
\psline[linecolor=black, linewidth=0.04](3.0985577,0.30000007)(2.2985578,-1.3)
\psline[linecolor=black, linewidth=0.04](2.2985578,-1.3)(0.098557666,-0.8999999)
\psline[linecolor=black, linewidth=0.04, linestyle=dashed, dash=0.17638889cm 0.10583334cm](2.2985578,-1.3)(2.2985578,-2.6999998)
\psline[linecolor=black, linewidth=0.04](2.2985578,-1.3)(4.0985575,-1.0999999)
\psline[linecolor=black, linewidth=0.04](4.0985575,-1.0999999)(5.2985578,0.100000076)
\psline[linecolor=black, linewidth=0.04, linestyle=dashed, dash=0.17638889cm 0.10583334cm](5.2985578,0.100000076)(7.2985578,-1.3)
\psline[linecolor=black, linewidth=0.04](5.2985578,0.100000076)(4.8985577,2.5)
\psline[linecolor=black, linewidth=0.04](4.8985577,2.5)(4.2985578,1.3000001)
\psline[linecolor=black, linewidth=0.04](4.2985578,1.3000001)(3.0985577,0.30000007)
\psline[linecolor=black, linewidth=0.04](3.0985577,0.30000007)(5.2985578,0.100000076)
\psline[linecolor=black, linewidth=0.04, linestyle=dashed, dash=0.17638889cm 0.10583334cm](2.6985576,1.7)(0.098557666,-0.8999999)
\psdots[linecolor=black, dotstyle=o, dotsize=0.2](3.0985577,2.7)
\psdots[linecolor=black, dotstyle=o, dotsize=0.2](0.29855767,1.1)
\psdots[linecolor=black, dotstyle=o, dotsize=0.2](4.2985578,1.3000001)
\psdots[linecolor=black, dotstyle=o, dotsize=0.2](2.6985576,1.7)
\psdots[linecolor=black, dotstyle=o, dotsize=0.2](0.6985577,-2.3)
\psdots[linecolor=black, dotstyle=o, dotsize=0.2](2.2985578,-1.3)
\psdots[linecolor=black, dotstyle=o, dotsize=0.2](5.2985578,0.100000076)
\psdots[linecolor=black, dotstyle=o, dotsize=0.2](6.698558,1.3000001)
\psdots[linecolor=black, dotstyle=o, dotsize=0.2](5.2985578,-2.6999998)
\end{pspicture}
}
\end{center}
\caption{Example of a non-full submap (solid edges) of a bipartite map (both solid and dashed edges).}
\label{figsubmap}
\end{figure}
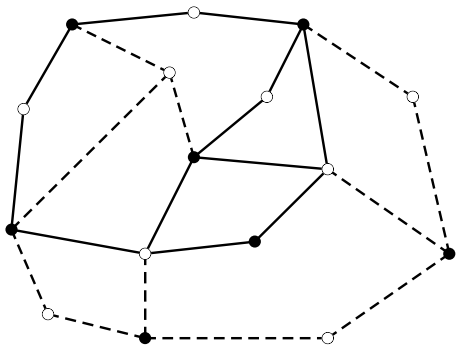

Then, we can give a relation between the degree of black (resp. white) vertices, the degree of faces and the number of edges of a bipartite map.

\begin{propdegbwvf}
\label{propdegbwvf}
Let $G$ be a bipartite map, then
\begin{displaymath}
\sum_{w \in \mathcal{W}_G}{\deg_G(w)} = \sum_{v \in \mathcal{V}_G}{\deg_G(v)} = \sum_{f \in \mathcal{F}_G}{\deg_G(f)} = E_G.
\end{displaymath}
\end{propdegbwvf}

\begin{proof}
As $G$ is bipartite, each edge of $G$ is incident to a unique black (resp. white) vertex, and so we deduce immediately the first relations:
\begin{displaymath}
\sum_{w \in \mathcal{W}_G}{\deg_G(w)} = E_G = \sum_{v \in \mathcal{V}_G}{\deg_G(v)}.
\end{displaymath}
To prove the remaining relation, it suffices to see that the degree of a face $f$ is equals to half of the number of edges incident to $f$, counted with multiplicity. Then, as each edge is incident to exactly two faces (maybe the same), we deduce the required relation.
\end{proof}

\textbf{Rk:} this property is also true if $G$ is a disjoint union of bipartite maps.

\subsubsection*{Union and intersection}

Let $H_1$ and $H_2$ be two submaps of a bipartite map $G$ having at least one common vertex. We define the union $H_1 \cup H_2$ of $H_1$ and $H_2$ to be the unique submap of $G$ constructed by erasing all the vertices, resp. edges, of $G$ that do not belong to $\mathcal{V}_{H_1} \cup \mathcal{V}_{H_2}$, resp. $\mathcal{E}_{H_1} \cup \mathcal{E}_{H_2}$. In particular, $H_1$, resp. $H_2$, is a submap of $H_1 \cup H_2$.

Similarly, we define the intersection $H_1 \cap H_2$ of $H_1$ and $H_2$ to be the submap (or the disjoint union of submaps) of $G$ constructed by erasing all the vertices, resp. edges, of $G$ that do not belong to $\mathcal{V}_{H_1} \cap \mathcal{V}_{H_2}$, resp. $\mathcal{E}_{H_1} \cap \mathcal{E}_{H_2}$. In particular, each connected component of $H_1 \cap H_2$ is a submap of both $H_1$ and $H_2$.

Finally, we give some relations between the number of faces and black vertices of these different maps. But first, we need to define these numbers. Consider two submaps $H_1$ and $H_2$ as previously, and denote by $L_1 , \ldots , L_n$ the connected components of $H_1 \cap H_2$. Then the number of faces of $H_1 \cap H_2$ is equal to
\begin{displaymath}
F_{H_1 \cap H_2} = \left( \sum_{i=1}^{n}{F_{L_i} - 1} \right) + 1.
\end{displaymath}
In other words, $F_{H_1 \cap H_2}$ is the number of connected components of $\mathbb{S}^2 \setminus (H_1 \cap H_2)$. Similarly, the number of black vertices of $H_1 \cap H_2$ is equal to
\begin{displaymath}
V_{H_1 \cap H_2} = \sum_{i=1}^{n}{V_{L_i}}.
\end{displaymath}

\begin{propdegunin}
\label{propdegunin}
Let $H_1$ and $H_2$ be two submaps of a bipartite map $G$, such that $H_1 \cap H_2$ is not empty. Then,
\begin{align*}
F_{H_1 \cup H_2} & = F_{H_1} + F_{H_2} - F_{H_1 \cap H_2} + (n-1) \\
V_{H_1 \cup H_2} & = V_{H_1} + V_{H_2} - V_{H_1 \cap H_2},
\end{align*}
where $n$ is equal to the number of connected components of $H_1 \cap H_2$.
\end{propdegunin}

The proof of this property is not difficult, and is left to the reader. Notice that this result can be generalized to the case where $H_1$ and $H_2$ are disjoint, i.e. $H_1 \cap H_2 = \emptyset$, by imposing that $F_{H_1 \cap H_2} = V_{H_1 \cap H_2} = 0$ and $n=0$. In particular, we can generalize the Euler's characteristic for a finite union of disjoint planar maps as follows.

\begin{theuler2}
\label{theuler2}
Let $G$ be the union of $n$ disjoint planar maps, for $n \geq 1$. Then
\begin{displaymath}
V_G + F_G = E_G + 2 + (n-1).
\end{displaymath}
\end{theuler2}

\subsubsection*{Complements}

Let $G$ be a bipartite map, and $H \subset G$, $H \neq G$. We consider the two following sets
\begin{itemize}
\item $\mathcal{M}_c(H)$ is the set of (disjoint union of) submap(s) $K$ of $G$ such that $H \cap K = \emptyset$.
\item $\mathcal{M}^c(H)$ is the set of (disjoint union of) submap(s) $K$ of $G$ such that $H \cup K = G$.
\end{itemize} 
Then, we define the \textbf{strict complement} of $H$, denoted $H_c$, by
\begin{displaymath}
H_c = \bigcup_{K \in \mathcal{M}_c(H)}{K}.
\end{displaymath}
In other words, $H_c$ is a disjoint union of submaps\footnote{$H_c$ can have multiple connected components} of $G$ obtained by erasing the vertices and edges of $H$, as well as the edges with at least one tip in $\mathcal{V}_H \cup \mathcal{W}_H$.

Similarly, we define the \textbf{relative complement} of $H$, denoted $H^c$, by
\begin{displaymath}
H^c = \bigcap_{K \in \mathcal{M}^c(H)}{K}.
\end{displaymath}
So, $H^c$ is the same as $H_c$ with the addition of the edges of $G$ not in $H$ but with at least one tip in $\mathcal{V}_H \cup \mathcal{W}_H$, and the tips of these edges  (see Figure \ref{figcsubmap} for an example). In particular, each connected component of $H_c$ is a submap of a connected component of $H^c$. 

\begin{figure}[ht]
\begin{center}
\psscalebox{0.8 0.8} 
{
\begin{pspicture}(0,-2.8006945)(7.397115,2.8006945)
\psdots[linecolor=black, dotsize=0.2](1.0985577,2.5)
\psdots[linecolor=black, dotsize=0.2](4.8985577,2.5)
\psdots[linecolor=black, dotsize=0.2](3.0985577,0.30000007)
\psdots[linecolor=black, dotsize=0.2](4.0985575,-1.0999999)
\psdots[linecolor=black, dotsize=0.2](7.2985578,-1.3)
\psdots[linecolor=black, dotsize=0.2](2.2985578,-2.6999998)
\psdots[linecolor=black, dotsize=0.2](0.098557666,-0.8999999)
\psline[linecolor=black, linewidth=0.04, linestyle=dashed, dash=0.17638889cm 0.10583334cm](1.0985577,2.5)(0.29855767,1.1)
\psline[linecolor=black, linewidth=0.04, linestyle=dashed, dash=0.17638889cm 0.10583334cm](0.29855767,1.1)(0.098557666,-0.8999999)
\psline[linecolor=black, linewidth=0.04](0.098557666,-0.8999999)(0.6985577,-2.3)
\psline[linecolor=black, linewidth=0.04](0.6985577,-2.3)(2.2985578,-2.6999998)
\psline[linecolor=black, linewidth=0.04](2.2985578,-2.6999998)(5.2985578,-2.6999998)
\psline[linecolor=black, linewidth=0.04, linestyle=dashed, dash=0.17638889cm 0.10583334cm](5.2985578,-2.6999998)(7.2985578,-1.3)
\psline[linecolor=black, linewidth=0.04, linestyle=dashed, dash=0.17638889cm 0.10583334cm](7.2985578,-1.3)(6.698558,1.3000001)
\psline[linecolor=black, linewidth=0.04, linestyle=dashed, dash=0.17638889cm 0.10583334cm](6.698558,1.3000001)(4.8985577,2.5)
\psline[linecolor=black, linewidth=0.04, linestyle=dashed, dash=0.17638889cm 0.10583334cm](4.8985577,2.5)(3.0985577,2.7)
\psline[linecolor=black, linewidth=0.04, linestyle=dashed, dash=0.17638889cm 0.10583334cm](3.0985577,2.7)(1.0985577,2.5)
\psline[linecolor=black, linewidth=0.04, linestyle=dashed, dash=0.17638889cm 0.10583334cm](1.0985577,2.5)(2.6985576,1.7)
\psline[linecolor=black, linewidth=0.04, linestyle=dashed, dash=0.17638889cm 0.10583334cm](2.6985576,1.7)(3.0985577,0.30000007)
\psline[linecolor=black, linewidth=0.04](3.0985577,0.30000007)(2.2985578,-1.3)
\psline[linecolor=black, linewidth=0.04](2.2985578,-1.3)(0.098557666,-0.8999999)
\psline[linecolor=black, linewidth=0.04](2.2985578,-1.3)(2.2985578,-2.6999998)
\psline[linecolor=black, linewidth=0.04](2.2985578,-1.3)(4.0985575,-1.0999999)
\psline[linecolor=black, linewidth=0.04](4.0985575,-1.0999999)(5.2985578,0.100000076)
\psline[linecolor=black, linewidth=0.04, linestyle=dashed, dash=0.17638889cm 0.10583334cm](5.2985578,0.100000076)(7.2985578,-1.3)
\psline[linecolor=black, linewidth=0.04](5.2985578,0.100000076)(4.8985577,2.5)
\psline[linecolor=black, linewidth=0.04](4.8985577,2.5)(4.2985578,1.3000001)
\psline[linecolor=black, linewidth=0.04](4.2985578,1.3000001)(3.0985577,0.30000007)
\psline[linecolor=black, linewidth=0.04](3.0985577,0.30000007)(5.2985578,0.100000076)
\psline[linecolor=black, linewidth=0.04, linestyle=dashed, dash=0.17638889cm 0.10583334cm](2.6985576,1.7)(0.098557666,-0.8999999)
\psdots[linecolor=black, dotstyle=o, dotsize=0.2](3.0985577,2.7)
\psdots[linecolor=black, dotstyle=o, dotsize=0.2](0.29855767,1.1)
\psdots[linecolor=black, dotstyle=o, dotsize=0.2](4.2985578,1.3000001)
\psdots[linecolor=black, dotstyle=o, dotsize=0.2](2.6985576,1.7)
\psdots[linecolor=black, dotstyle=o, dotsize=0.2](0.6985577,-2.3)
\psdots[linecolor=black, dotstyle=o, dotsize=0.2](2.2985578,-1.3)
\psdots[linecolor=black, dotstyle=o, dotsize=0.2](5.2985578,0.100000076)
\psdots[linecolor=black, dotstyle=o, dotsize=0.2](6.698558,1.3000001)
\psdots[linecolor=black, dotstyle=o, dotsize=0.2](5.2985578,-2.6999998)
\end{pspicture}
}
\end{center}
\caption{Example of a full submap (solid edges) of a bipartite map, and its relative complement (dashed edges).}
\label{figcsubmap}
\end{figure}
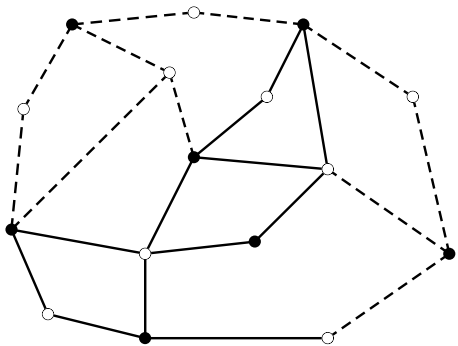

\begin{propconnconv}
\label{propconnconv}
Let $G$ be a bipartite map, and $H$ a submap of $G$. Suppose that $H$ is not a full submap of $G$, then there exists a finite number of full submaps $H_i$ of $G$ such that
\begin{displaymath}
\bigcap_{i}{H_i} = H \quad \text{and} \quad H_i \cup H_j = G \text{ for any pair } i \neq j.
\end{displaymath}
\end{propconnconv}

The minimum number of full submaps $H_i$ we need to obtain the equality is called the \textbf{genus} of $H$. By convention, the genus of a full submap of $G$ is equal to $1$.

\begin{proof}
We denote by $\mathcal{F}$ the set of faces of $H$ not in $\mathcal{F}_G$. Then for each $f \in \mathcal{F}$, we construct a submap $K_f$ of $G$ by erasing each edge and vertex of $G$ except those contained in the face $f$ and its boundary. 
Then we construct the planar map $H_f$ as follows:
\begin{displaymath}
H_f = H \bigcup_{g \in \mathcal{F} \atop g \neq f}{K_{g}}.
\end{displaymath}
Clearly, $H$ is a submap of $H_f$, and $H_f$ is a full submap of $G$. The rest is a direct consequence of the construction of the submaps $H_f$, and is left to the reader.
\end{proof}

\subsection{Hall's marriage theorem}

Let $\mathcal{A}$ be a set of elements, and $\mathcal{S}$ a finite family of finite subsets $\{A_1 , A_2 , \ldots \}$ of $\mathcal{A}$.

\begin{defSDR}
$\mathcal{S}$ has a \textbf{system of distinct representatives} if there exists distinct elements $\{x_1 , x_2 , \ldots \}$ such that $x_i \in A_i$, for each $i=1,2, \ldots$

If furthermore $card(\mathcal{S}) = card \left( \bigcup A_i \right)$, we say that $\mathcal{S}$ has a \textbf{perfect matching}.
\end{defSDR}

Now, we have to determine a condition such that a family of finite sets has a system of distinct representatives.

\begin{defmarcond}
We say that $\mathcal{S}$ satisfies the \textbf{marriage condition} if the union of any $k$ of the sets $A_i$ contains at least $k$ elements for $k = 1,2, \ldots$
\end{defmarcond}

\begin{thhall}
\label{thhall}
(Hall, \cite{hall}) $\mathcal{S}$ has a system of distinct representatives if and only if $\mathcal{S}$ satisfies the marriage condition.
\end{thhall}

\begin{proof}
If $\mathcal{S}$ has a system of distinct representatives, then it's evident that $\mathcal{S}$ satisfies the marriage condition.

Now we prove that the converse is true by induction. If the family $\mathcal{S}$ contains only one finite set, the result is obvious. Suppose that the result of the Theorem \ref{thhall} is true for any family containing at most $n$ finite sets, and consider a family $\mathcal{S}$ of $(n+1)$ finite sets $\{A_1 , \ldots , A_{n+1}\}$ which satisfies the marriage condition. Then two situations can appear:
\begin{itemize}
\item The union of any $k$ of the sets of $\mathcal{S}$ contains at least $(k+1)$ elements for $k=1, \ldots , n$. In this case, we can choose any element of $A_{n+1}$ as representative of this set, and construct the family $\tilde{\mathcal{S}} = \{B_1 , \ldots , B_n\}$, where
\begin{displaymath}
B_i = \left\{ \begin{array}{ll}
A_i \setminus \{ x_{n+1} \} & \text{if } x_{n+1} \in A_i \\
A_i & \text{else}
\end{array} . \right.
\end{displaymath}
Then $\tilde{\mathcal{S}}$ is a family of $n$ finite sets which satisfies the marriage condition. So by induction, $\tilde{\mathcal{S}}$ has a system of distinct representatives, and we deduce that $\mathcal{S}$ has also a system of distinct representatives.

\item There exists a subfamily $F$ of $\mathcal{S}$, $F \neq \mathcal{S}$, such that there are as many finite sets in $F$ as elements in the union of the sets of $F$. Up to change the order of the sets in $\mathcal{S}$, we can suppose that $F = \{A_1 , \ldots , A_k\}$. In this case, as $F$ is a subfamily of $\mathcal{S}$, then $F$ satisfies the marriage condition, so by induction $F$ has a system of distinct representatives. We denote by $(x_1 , \ldots , x_n)$ the distinct representatives of $F$. Now, consider the subfamily $G = \{A_{k+1} , \ldots , A_{n+1}\}$, and construct a new family $\tilde{G} = \{B_{k+1} , \ldots , B_{n+1}\}$ where $B_i = A_i \setminus \{x_1 , \ldots , x_k\}$.

Then it's easy to see that $\tilde{G}$ satisfies the marriage condition. In fact, suppose $\tilde{G}$ does not satisfy the marriage condition, so there exists $l$ sets of $\tilde{G}$ such that the union of these $l$ sets contains at most $(l-1)$ elements. Once again, up to change the order of the sets of $\tilde{G}$, we can suppose that these $l$ sets are $B_{k+1} , \ldots , B_{k+l}$. In this configuration, we deduce that the union of the $k+l$ sets $A_1 , \ldots , A_{k+l}$ contains at most $(k+l-1)$ elements, but this is in contradiction this the fact that $\mathcal{S}$ satisfies the marriage condition.

In conclusion, $\tilde{G}$ satisfies the marriage condition, so it has a system of distinct representatives, and the union of the distinct representatives of $F$ and $\tilde{G}$ gives a system of distinct representatives for $\mathcal{S}$.
\end{itemize}
\end{proof}

\section{Balanced maps}

The objective of this section is to characterize combinatorially self branched coverings of the sphere by means of a specific kind of maps, namely the balanced maps. The precise statement will be given in Theorem \ref{thbalmap}.

\subsection{Construction of a map from a branched covering}
\label{construction}

Let $\pi : \mathbb{S}^2 \rightarrow \mathbb{S}^2$ be a branched covering of the sphere of degree $d$. Let $A$ be a set of $n$ distinct points consisting of (or more generally containing) the critical values. Up to postcomposition by an orientation preserving homeomorphism, we can suppose that this set of $n$ points is located at the $n$-th roots of unity. 
We represent each point by a white dot, and label these dots from $1$ to $n$ such that $exp(2i\pi k / n)$ gets the label $k$. Place a black dot at the origin, and trace a radial edge between this black dot and each white dot, so we obtain a connected map. Moreover, this map, denoted by $T$, is an increasing bipartite map (and more precisely a star-like map since $T$ has only one face and one black vertex).

\begin{figure}[ht]
\begin{center}
\psscalebox{0.8 0.8} 
{
\begin{pspicture}(0,-3.064247)(16.940695,3.064247)
\pscircle[linecolor=black, linewidth=0.04, fillstyle=solid,fillcolor=black, dimen=outer](1.404,2.9459777){0.12}
\pscircle[linecolor=black, linewidth=0.04, fillstyle=solid,fillcolor=black, dimen=outer](0.604,-1.4540223){0.12}
\pscircle[linecolor=black, linewidth=0.04, fillstyle=solid,fillcolor=black, dimen=outer](2.604,0.5459777){0.12}
\pscircle[linecolor=black, linewidth=0.04, fillstyle=solid,fillcolor=black, dimen=outer](4.206,1.7459778){0.12}
\pscircle[linecolor=black, linewidth=0.04, fillstyle=solid,fillcolor=black, dimen=outer](5.805,-0.6540223){0.12}
\pscircle[linecolor=black, linewidth=0.04, fillstyle=solid,fillcolor=black, dimen=outer](7.404,-1.4540223){0.12}
\pscircle[linecolor=black, linewidth=0.04, fillstyle=solid,fillcolor=black, dimen=outer](3.404,-2.2540224){0.12}
\pscircle[linecolor=black, linewidth=0.04, dimen=outer](1.805,-2.2540224){0.12}
\pscircle[linecolor=black, linewidth=0.04, dimen=outer](5.805,-2.2540224){0.12}
\pscircle[linecolor=black, linewidth=0.04, dimen=outer](3.805,-0.6540223){0.12}
\pscircle[linecolor=black, linewidth=0.04, dimen=outer](2.604,1.7459778){0.12}
\pscircle[linecolor=black, linewidth=0.04, dimen=outer](0.604,0.9459777){0.12}
\pscircle[linecolor=black, linewidth=0.04, dimen=outer](1.805,-0.6540223){0.12}
\pscircle[linecolor=black, linewidth=0.04, dimen=outer](5.805,0.9459777){0.12}
\pscircle[linecolor=black, linewidth=0.04, dimen=outer](6.607,2.1459777){0.12}
\psline[linecolor=black, linewidth=0.04](1.404,2.9659777)(0.64069444,1.0459777)
\psline[linecolor=black, linewidth=0.04](1.4206945,2.9459777)(2.5206945,1.8259777)
\psline[linecolor=black, linewidth=0.04](1.404,2.9459777)(6.4606943,2.1659777)
\psline[linecolor=black, linewidth=0.04](4.206,1.7659777)(2.7206945,1.7659777)
\psline[linecolor=black, linewidth=0.04](4.206,1.7459778)(5.6806946,1.0059777)
\psline[linecolor=black, linewidth=0.04](2.5806944,0.58597773)(2.604,1.6459777)
\psline[linecolor=black, linewidth=0.04](2.604,0.58597773)(1.8606944,-0.5540223)
\psline[linecolor=black, linewidth=0.04](2.604,0.52597773)(3.703,-0.5740223)
\psline[linecolor=black, linewidth=0.04](3.404,-2.2340224)(3.7606945,-0.7540223)
\psline[linecolor=black, linewidth=0.04](3.4206944,-2.2140224)(5.6806946,-2.2340224)
\psline[linecolor=black, linewidth=0.04](3.3806944,-2.2540224)(1.904,-2.2540224)
\psline[linecolor=black, linewidth=0.04](0.604,-1.4740223)(1.6806945,-2.1940222)
\psline[linecolor=black, linewidth=0.04](0.604,-1.4340223)(1.7206944,-0.7140223)
\psline[linecolor=black, linewidth=0.04](0.58069444,-1.4140223)(0.58069444,0.8659777)
\psline[linecolor=black, linewidth=0.04](5.805,-0.6340223)(5.7806945,0.8459777)
\psline[linecolor=black, linewidth=0.04](5.805,-0.6740223)(3.904,-0.6340223)
\psline[linecolor=black, linewidth=0.04](5.805,-0.6740223)(5.805,-2.1540222)
\psline[linecolor=black, linewidth=0.04](7.3806944,-1.4340223)(5.904,-2.2340224)
\psline[linecolor=black, linewidth=0.04](7.3806944,-1.3940223)(6.6606946,2.0659778)
\pscircle[linecolor=black, linewidth=0.04, dimen=outer](1.6206944,1.8259777){0.1}
\pscircle[linecolor=black, linewidth=0.04, dimen=outer](2.404,2.2459776){0.1}
\pscircle[linecolor=black, linewidth=0.04, dimen=outer](2.4406943,2.5859778){0.1}
\pscircle[linecolor=black, linewidth=0.04, dimen=outer](4.2406945,2.1659777){0.1}
\pscircle[linecolor=black, linewidth=0.04, dimen=outer](3.604,1.3859777){0.1}
\pscircle[linecolor=black, linewidth=0.04, dimen=outer](4.0,1.1059777){0.1}
\pscircle[linecolor=black, linewidth=0.04, dimen=outer](4.4206944,1.1659777){0.1}
\pscircle[linecolor=black, linewidth=0.04, dimen=outer](2.3606944,-0.1940223){0.1}
\pscircle[linecolor=black, linewidth=0.04, dimen=outer](2.6206944,-0.2140223){0.1}
\pscircle[linecolor=black, linewidth=0.04, dimen=outer](2.904,-0.2140223){0.1}
\pscircle[linecolor=black, linewidth=0.04, dimen=outer](0.7806944,-0.7340223){0.1}
\pscircle[linecolor=black, linewidth=0.04, dimen=outer](1.0806944,-0.8940223){0.1}
\pscircle[linecolor=black, linewidth=0.04, dimen=outer](0.10425,-2.0140224){0.1}
\pscircle[linecolor=black, linewidth=0.04, dimen=outer](3.3606944,-2.8740222){0.1}
\pscircle[linecolor=black, linewidth=0.04, dimen=outer](3.7206945,-1.7140223){0.1}
\pscircle[linecolor=black, linewidth=0.04, dimen=outer](3.8606944,-2.0340223){0.1}
\pscircle[linecolor=black, linewidth=0.04, dimen=outer](5.206,-0.97402227){0.1}
\pscircle[linecolor=black, linewidth=0.04, dimen=outer](5.5206943,-1.2940223){0.1}
\pscircle[linecolor=black, linewidth=0.04, dimen=outer](6.9806943,-0.6940223){0.1}
\pscircle[linecolor=black, linewidth=0.04, dimen=outer](6.6606946,-0.9540223){0.1}
\pscircle[linecolor=black, linewidth=0.04, dimen=outer](6.6206946,-1.3540223){0.1}
\pscircle[linecolor=black, linewidth=0.04, dimen=outer](6.6606946,-1.6140223){0.1}
\psline[linecolor=black, linewidth=0.04](5.8206944,-0.6340223)(6.560694,2.0659778)
\psline[linecolor=black, linewidth=0.07, arrowsize=0.053cm 2.0,arrowlength=1.4,arrowinset=0.0]{->}(9.020695,-0.2540223)(11.90,-0.23402229)
\pscircle[linecolor=black, linewidth=0.04, fillstyle=solid,fillcolor=black, dimen=outer](14.980695,-0.2540223){0.12}
\pscircle[linecolor=black, linewidth=0.04, dimen=outer](16.60,-0.27402228){0.12}
\pscircle[linecolor=black, linewidth=0.04, dimen=outer](15.80,0.9659777){0.12}
\pscircle[linecolor=black, linewidth=0.04, dimen=outer](14.160694,0.9859777){0.12}
\pscircle[linecolor=black, linewidth=0.04, dimen=outer](13.380694,-0.27402228){0.12}
\pscircle[linecolor=black, linewidth=0.04, dimen=outer](14.180695,-1.4340223){0.12}
\pscircle[linecolor=black, linewidth=0.04, dimen=outer](15.80,-1.4540223){0.12}
\psline[linecolor=black, linewidth=0.04](14.960694,-0.2540223)(16.480694,-0.2540223)
\psline[linecolor=black, linewidth=0.04](15.00,-0.1940223)(15.740694,0.9059777)
\psline[linecolor=black, linewidth=0.04](14.940695,-0.2140223)(14.220695,0.9059777)
\psline[linecolor=black, linewidth=0.04](14.980695,-0.2140223)(13.480695,-0.27402228)
\psline[linecolor=black, linewidth=0.04](14.960694,-0.2540223)(14.240694,-1.3140223)
\psline[linecolor=black, linewidth=0.04](14.980695,-0.23402229)(15.720695,-1.3740222)
\psline[linecolor=black, linewidth=0.04](1.404,2.9459777)(1.5806944,1.9259777)
\psline[linecolor=black, linewidth=0.04](1.4406945,2.9459777)(2.3406944,2.6259778)
\psline[linecolor=black, linewidth=0.04](1.404,2.9459777)(2.3206944,2.3059778)
\psline[linecolor=black, linewidth=0.04](4.206,1.7859777)(4.2206945,2.0859778)
\psline[linecolor=black, linewidth=0.04](4.1606946,1.7659777)(3.6606944,1.4459777)
\psline[linecolor=black, linewidth=0.04](4.1806946,1.7459778)(4.040694,1.1859777)
\psline[linecolor=black, linewidth=0.04](4.206,1.7459778)(4.3806944,1.2659777)
\psline[linecolor=black, linewidth=0.04](7.4206944,-1.4340223)(7.040694,-0.7540223)
\psline[linecolor=black, linewidth=0.04](7.3606944,-1.4340223)(6.7406945,-0.97402227)
\psline[linecolor=black, linewidth=0.04](7.3406944,-1.4540223)(6.706,-1.3340223)
\psline[linecolor=black, linewidth=0.04](7.3606944,-1.4340223)(6.7406945,-1.5940223)
\psline[linecolor=black, linewidth=0.04](5.7606945,-0.6740223)(5.2806945,-0.9340223)
\psline[linecolor=black, linewidth=0.04](5.7806945,-0.6340223)(5.560694,-1.2140223)
\psline[linecolor=black, linewidth=0.04](3.4206944,-2.2540224)(3.7806945,-2.0740223)
\psline[linecolor=black, linewidth=0.04](3.4206944,-2.1940222)(3.6606944,-1.7740223)
\psline[linecolor=black, linewidth=0.04](2.5606945,0.58597773)(2.3606944,-0.11402229)
\psline[linecolor=black, linewidth=0.04](2.5806944,0.5059777)(2.6206944,-0.11402229)
\psline[linecolor=black, linewidth=0.04](2.6206944,0.5459777)(2.8606944,-0.11402229)
\psline[linecolor=black, linewidth=0.04](0.56069446,-1.4140223)(0.16069442,-1.9340223)
\psline[linecolor=black, linewidth=0.04](3.3806944,-2.2540224)(3.3606944,-2.7740223)
\psline[linecolor=black, linewidth=0.04](0.58069444,-1.4340223)(1.0206944,-0.9540223)
\psline[linecolor=black, linewidth=0.04](0.604,-1.4140223)(0.76069444,-0.8140223)
\rput[bl](16.760695,-0.33402228){$6$}
\rput[bl](15.840694,1.1459777){$1$}
\rput[bl](13.960694,1.1459777){$2$}
\rput[bl](13.040694,-0.3540223){$3$}
\rput[bl](13.880694,-1.6740223){$4$}
\rput[bl](15.940695,-1.6340222){$5$}
\rput[bl](2.2806945,1.5259777){$3$}
\rput[bl](0.3044,0.8259777){$1$}
\rput[bl](1.7606944,-2.6740222){$3$}
\rput[bl](5.9406943,-2.5140224){$5$}
\rput[bl](5.8406944,1.0459777){$1$}
\rput[bl](6.706,2.2659776){$6$}
\rput[bl](3.9406943,-0.4740223){$2$}
\rput[bl](1.6406944,-0.4340223){$4$}
\rput[bl](4.3406944,2.0659778){$2$}
\rput[bl](3.3406944,1.1859777){$4$}
\rput[bl](3.8806944,0.7259777){$5$}
\rput[bl](4.540694,0.8859777){$6$}
\rput[bl](1.5406945,1.4659777){$2$}
\rput[bl](2.505,2.0859778){$4$}
\rput[bl](2.5406945,2.4059777){$5$}
\rput[bl](0.08069443,-2.3940222){$2$}
\rput[bl](1.1606945,-0.7540223){$5$}
\rput[bl](0.7806944,-0.6140223){$6$}
\rput[bl](2.2206945,-0.5940223){$5$}
\rput[bl](2.5806944,-0.6140223){$6$}
\rput[bl](2.9206944,-0.5940223){$1$}
\rput[bl](3.5206945,-2.9940224){$4$}
\rput[bl](3.9806945,-1.9940223){$6$}
\rput[bl](3.805,-1.5940223){$1$}
\rput[bl](4.9606943,-1.2540222){$3$}
\rput[bl](5.4206944,-1.6940223){$4$}
\rput[bl](6.9206944,-0.5340223){$1$}
\rput[bl](6.503,-0.8340223){$2$}
\rput[bl](6.3806944,-1.3140223){$3$}
\rput[bl](6.3406944,-1.7540222){$4$}
\rput[bl](9.980695,-0.054022294){\Large $\pi$}
\end{pspicture}
}
\end{center}
\caption{Example of the representation $\Gamma$ of a branched covering $\pi$ of degree $7$.}
\label{figrepresent}
\end{figure}
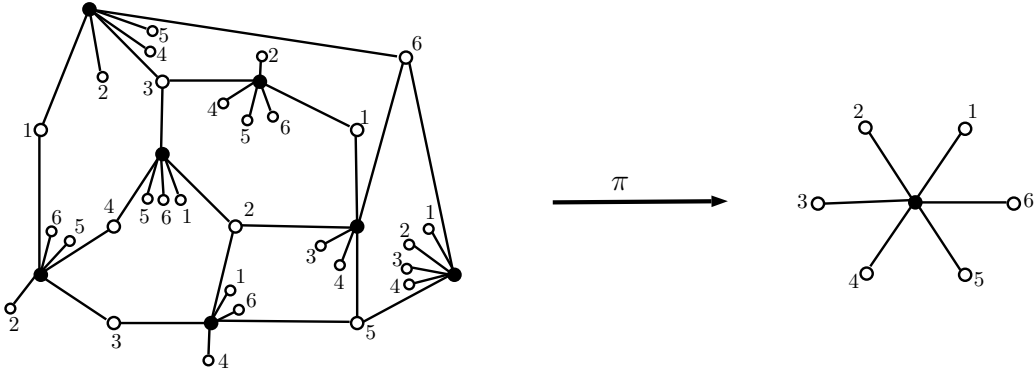

From now on, we are interested in the pull-back map of $T$ by $\pi$.
\begin{thgamcon}
\label{thgamcon} The map   $\Gamma := \pi^{-1}(T)$  (i.e. the cells are the connected components of pulled back cells) is an increasing bipartite map with exactly $d$ faces.
\end{thgamcon}

\begin{proof}
By construction, we prove easily that $\Gamma$ is an increasing bipartite map. So the only difficulty is to prove that $\Gamma$ is a map with $d$ faces. Note that the induced map $\pi: \mathbb{S}^2 \setminus \Gamma \rightarrow \mathbb{S}^2 \setminus T$ is a covering of degree $d$, and $\mathbb{S}^2 \setminus T$ is a simply connected open set. Fix a point $b \in \mathbb{S}^2 \setminus T$. Then, for each point $x \in \pi^{-1}(b)$, the continuous function $Id : \mathbb{S}^2 \setminus T \rightarrow \mathbb{S}^2 \setminus T$ has a (unique) lift $s: \mathbb{S}^2 \setminus T \rightarrow \mathbb{S}^2 \setminus \Gamma$ such that $s(b)=x$. Moreover, the function $s$ is a global section of $\pi$ and defines a homeomorphism of $\mathbb{S}^2 \setminus T$ on the connected component of $\mathbb{S}^2 \setminus \Gamma$ containing $x$. So $\mathbb{S}^2 \setminus \Gamma$ is the union of $d$ simply connected domains.

Moreover, the complement of a union of disjoint simply connected domains in $\mathbb{S}^2$ is necessarily connected. In conclusion, $\Gamma$ is an increasing bipartite map with $d$ faces.
\end{proof}

Now, let us give some details about the map $\Gamma$. By definition, for each point of the regular set of the branched covering $\pi$, there exists an open neighborhood $U$ such that the preimage of $U$ by $\pi$ is a disjoint union of $d$ open sets, each homeomorphic to $U$. In other words, each point of the sphere, except the critical values of $\pi$, has exactly $d$ distinct preimages by $\pi$. In particular, the black vertex and each edge of $T$ has $d$ preimages by $\pi$. So $\Gamma$ is an increasing bipartite map consisting of $d$ faces, $n\cdot d$ edges and $d$ black vertices. Moreover, the degree of each face and black vertex of $\Gamma$ is the same, equals to $n$.

Finally, as the passport of $\pi$ describes the behavior of each of its critical value, it's easy to see that the number of white vertices of $\Gamma$ is equal to the number of integers in the passport of $\pi$.

A direct consequence of this last theorem is the Riemann-Hurwitz condition that a passport must satisfy to be realizable. In fact, by Theorem \ref{thgamcon} and as $\pi$ is a branched covering of degree $d$, $\Gamma$ is a connected map with $d$ faces, $n \cdot d$ edges and $d$ black vertices. So if we denote by $m$ the number of white vertices of $\Gamma$, given by the passport of $\pi$, then by the Euler characteristic, we have
\begin{displaymath}
(d+m) + d = d \cdot n + 2
\end{displaymath} 
and so
\begin{displaymath}
2d - 2 = d \cdot n - m = \nu(\mathcal{D}(\pi)).
\end{displaymath}

\begin{defrepresentation}
We call $\Gamma$ a \textbf{representation} of the branched covering $\pi$.
\end{defrepresentation}

Thus we have a first model of a branched covering by a geometrical object. Nevertheless, this object is actually too complicated and so from now on, our goal is to simplify it without losing important information to be able to solve some realizability problems of given passports. For that, we use two maps in order to reduce the complexity of $\Gamma$, and then we prove in the next subsection that we can combine these two ideas.

\vspace{2mm}

First, let's see what happens if we erase each $1$-valence (automatically white) vertex and its attached edge of $\Gamma$. As $\Gamma$ is an increasing bipartite map with as many black vertices as faces, each of degree $n$, and with white vertices labelled from $1$ to $n$, then by erasing every $1$-valence vertex of $\Gamma$, we obtain a unique increasing bipartite map with the same properties as $\Gamma$ except those about the degree of each black vertex and face.

Conversely, if we consider such a map, it's simple to reconstruct our representation $\Gamma$ (cf Figure \ref{figrepresent1} for an example). In conclusion, we have a first bijection that allows us to simplify our representation. We call this map the \textbf{labeled skeleton} of the branched covering $\pi$.

\begin{figure}[ht]
\begin{center}
\psscalebox{0.8 0.8} 
{
\begin{pspicture}(0,-2.9042468)(7.218269,2.9042468)
\pscircle[linecolor=black, linewidth=0.04, fillstyle=solid,fillcolor=black, dimen=outer](1.1,2.7859776){0.12}
\pscircle[linecolor=black, linewidth=0.04, fillstyle=solid,fillcolor=black, dimen=outer](0.3,-1.6140223){0.12}
\pscircle[linecolor=black, linewidth=0.04, fillstyle=solid,fillcolor=black, dimen=outer](2.3,0.38597772){0.12}
\pscircle[linecolor=black, linewidth=0.04, fillstyle=solid,fillcolor=black, dimen=outer](3.9,1.5859777){0.12}
\pscircle[linecolor=black, linewidth=0.04, fillstyle=solid,fillcolor=black, dimen=outer](5.5,-0.8140223){0.12}
\pscircle[linecolor=black, linewidth=0.04, fillstyle=solid,fillcolor=black, dimen=outer](7.1,-1.6140223){0.12}
\pscircle[linecolor=black, linewidth=0.04, fillstyle=solid,fillcolor=black, dimen=outer](3.1,-2.4140222){0.12}
\pscircle[linecolor=black, linewidth=0.04, dimen=outer](1.5,-2.4140222){0.12}
\pscircle[linecolor=black, linewidth=0.04, dimen=outer](5.5,-2.4140222){0.12}
\pscircle[linecolor=black, linewidth=0.04, dimen=outer](3.5,-0.8140223){0.12}
\pscircle[linecolor=black, linewidth=0.04, dimen=outer](2.3,1.5859777){0.12}
\pscircle[linecolor=black, linewidth=0.04, dimen=outer](0.3,0.7859777){0.12}
\pscircle[linecolor=black, linewidth=0.04, dimen=outer](1.5,-0.8140223){0.12}
\pscircle[linecolor=black, linewidth=0.04, dimen=outer](5.5,0.7859777){0.12}
\pscircle[linecolor=black, linewidth=0.04, dimen=outer](6.3,1.9859776){0.12}
\psline[linecolor=black, linewidth=0.04](1.1,2.8059778)(0.34,0.8859777)
\psline[linecolor=black, linewidth=0.04](1.12,2.7859776)(2.22,1.6659777)
\psline[linecolor=black, linewidth=0.04](1.1,2.7859776)(6.16,2.0059776)
\psline[linecolor=black, linewidth=0.04](3.9,1.6059777)(2.42,1.6059777)
\psline[linecolor=black, linewidth=0.04](3.9,1.5859777)(5.38,0.8459777)
\psline[linecolor=black, linewidth=0.04](2.28,0.4259777)(2.3,1.4859776)
\psline[linecolor=black, linewidth=0.04](2.3,0.4259777)(1.56,-0.7140223)
\psline[linecolor=black, linewidth=0.04](2.3,0.3659777)(3.4,-0.7340223)
\psline[linecolor=black, linewidth=0.04](3.1,-2.3940222)(3.46,-0.91402227)
\psline[linecolor=black, linewidth=0.04](3.12,-2.3740222)(5.38,-2.3940222)
\psline[linecolor=black, linewidth=0.04](3.08,-2.4140222)(1.6,-2.4140222)
\psline[linecolor=black, linewidth=0.04](0.3,-1.6340222)(1.38,-2.3540223)
\psline[linecolor=black, linewidth=0.04](0.3,-1.5940223)(1.42,-0.8740223)
\psline[linecolor=black, linewidth=0.04](0.28,-1.5740223)(0.28,0.7059777)
\psline[linecolor=black, linewidth=0.04](5.5,-0.7940223)(5.48,0.6859777)
\psline[linecolor=black, linewidth=0.04](5.5,-0.8340223)(3.6,-0.7940223)
\psline[linecolor=black, linewidth=0.04](5.5,-0.8340223)(5.5,-2.3140223)
\psline[linecolor=black, linewidth=0.04](7.08,-1.5940223)(5.6,-2.3940222)
\psline[linecolor=black, linewidth=0.04](7.08,-1.5540223)(6.36,1.9059777)
\psline[linecolor=black, linewidth=0.04](5.52,-0.7940223)(6.26,1.9059777)
\rput[bl](1.98,1.3659778){$3$}
\rput[bl](0.0,0.6659777){$1$}
\rput[bl](1.46,-2.8340223){$3$}
\rput[bl](5.64,-2.6740222){$5$}
\rput[bl](5.54,0.8859777){$1$}
\rput[bl](6.4,2.1059778){$6$}
\rput[bl](3.64,-0.6340223){$2$}
\rput[bl](1.34,-0.5940223){$4$}
\end{pspicture}
}
\end{center}
\caption{Example of the increasing bipartite map obtained by erasing all the $1$-degree vertices of $\Gamma$.}
\label{figrepresent1}
\end{figure}
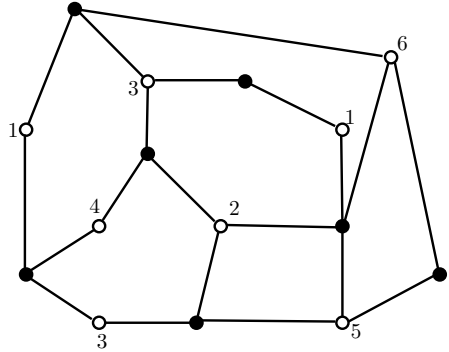

\vspace{2mm}

An alternative possibility is to erase the label of each white vertex of $\Gamma$ such that we obtain a unique bipartite map, denoted by $\tilde{\Gamma}$, with as many black vertices as faces, each of degree $n$. We call these maps a $n$\textbf{-regular skeleton} of the branched covering. Conversely, consider such a map $\tilde{\Gamma}$, we prove the following proposition:

\begin{proprepre}
Every bipartite map with equal number $d$ of faces and black vertices, each of constant degree $n$, comes from a representation of a branched covering of degree $d$, with at most $n$ distinct critical values.
\end{proprepre}

For that, we first need to transform $\tilde{\Gamma}$ by changing each black vertex and the $n$ edges coming from this vertex by a black $n$-gon (cf Figure \ref{figrepresent2} for an example), so we obtain a $n$-regular map, i.e. a map in which each face is of degree $n$, denoted by $\Gamma_r$. Notice that $\Gamma_r$ has two kinds of faces (those containing black vertices and those that are blank) and the two kinds alternate like a chessboard. So we can orient the edges of $\Gamma_r$ keeping  black faces to the left.

Now, choose a vertex $v$ of $\Gamma_r$ (i.e. a white vertex of $\Gamma$), and label it by $0$. Then we label all other vertices by the directed edge-distance from $v$ modulo $n$, i.e. the number of edges we need to pass through following the edge-orientation of $\Gamma_r$. Notice that this labeling depends on the choice of the first vertex $v$.

\begin{proplabel}
\label{proplabel}
This labeling is well-defined \footnote{One can prove easily this result using an argument of cohomology, but here we choose to avoid such a theory.}.
\end{proplabel}

\begin{proof}
We have to prove that for any given vertex $w$ of $\Gamma_r$, the label of $w$ is independent of the choice of a  directed path from $v$ to $w$, i.e. two paths from $v$ to $w$ have the same length modulo $n$. Consider two directed paths, denoted by $\gamma_1$ and $\gamma_2$, from $v$ to $w$, and a directed path $\gamma$ from $w$ to $v$, preserving the orientation of $\Gamma_r$, so that the concatenation $\gamma \circ \gamma_i$ of $\gamma$ and $\gamma_i$ gives a closed path, for $i=1$, $2$. 

Let us fix an index $i=1$ or $2$. If $\gamma \cap \gamma_i \neq \{v,w\}$, then $\gamma \circ \gamma_i$ is a 'gluing' of some Jordan curves, and so up to study each Jordan curves separately, we can suppose that $\gamma \cap \gamma_i = \{v,w\}$, i.e. $\gamma \circ \gamma_i$ defines a Jordan curve. In consequence, $\gamma \circ \gamma_i$ divides the set $\mathcal{F}_{\Gamma_r}$ into two subsets: one containing the faces to the left of $\gamma \circ \gamma_i$, and the other containing the faces to the right of $\gamma \circ \gamma_i$.

For the following, we consider the function
\begin{displaymath}
\omega : \begin{array}{ccc}
\mathcal{F}_{\Gamma_r} & \rightarrow & \mathbb{R}[\mathcal{E}_{\Gamma_r}] \\
f & \mapsto & \varepsilon_f \cdot \sum \limits_{a \in \mathcal{E}_{\Gamma_r} \cap \partial f}{a}
\end{array},
\end{displaymath}
with $\varepsilon_f = 1$ if $f$ is a black face, and $\varepsilon_f = -1$ else. Then, if we denote by $\omega(\gamma \circ \gamma_i)$ the sum of all the edges (with coefficient $1$) passed through by $\gamma \circ \gamma_i$, and by $\mathcal{F}_{ext}$ the set of the faces to the left of $\gamma \circ \gamma_i$, we prove easily that:
\begin{displaymath}
\omega(\gamma \circ \gamma_i) = \sum_{f \in \mathcal{F}_{ext}}{\omega(f)} .
\end{displaymath}
This last result implies that the length of the path $\gamma \circ \gamma_i$ (modulo $n$) is equal to
\begin{displaymath}
\sum_{f \in \mathcal{F}_{ext}} ~ \sum_{a \in \mathcal{E}_{\Gamma_r} \cap \partial f}{1}.
\end{displaymath}
As each face of $\Gamma_r$ is of degree $n$, we deduce that the length of $\gamma \circ \gamma_i$ is equal to $0$ modulo $n$, i.e. the length of $\gamma$ plus the length of $\gamma_i$ is equal to $0$ modulo $n$. 

In conclusion, $\gamma_1$ and $\gamma_2$ have the same length modulo $n$. 
\end{proof}

\begin{figure}[ht]
\begin{center}
\psscalebox{0.8 0.8} 
{
\begin{pspicture}(0,-3.019482)(7.5189633,3.019482)
\pscircle[linecolor=black, linewidth=0.04, fillstyle=solid,fillcolor=black, dimen=outer](1.404,2.9012127){0.12}
\pscircle[linecolor=black, linewidth=0.04, fillstyle=solid,fillcolor=black, dimen=outer](0.604,-1.4987873){0.12}
\pscircle[linecolor=black, linewidth=0.04, fillstyle=solid,fillcolor=black, dimen=outer](2.604,0.5012127){0.12}
\pscircle[linecolor=black, linewidth=0.04, fillstyle=solid,fillcolor=black, dimen=outer](4.206,1.7012126){0.12}
\pscircle[linecolor=black, linewidth=0.04, fillstyle=solid,fillcolor=black, dimen=outer](5.805,-0.69878733){0.12}
\pscircle[linecolor=black, linewidth=0.04, fillstyle=solid,fillcolor=black, dimen=outer](7.404,-1.4987873){0.12}
\pscircle[linecolor=black, linewidth=0.04, fillstyle=solid,fillcolor=black, dimen=outer](3.404,-2.2987874){0.12}
\pscircle[linecolor=black, linewidth=0.04, dimen=outer](1.805,-2.2987874){0.12}
\pscircle[linecolor=black, linewidth=0.04, dimen=outer](5.805,-2.2987874){0.12}
\pscircle[linecolor=black, linewidth=0.04, dimen=outer](3.805,-0.69878733){0.12}
\pscircle[linecolor=black, linewidth=0.04, dimen=outer](2.604,1.7012126){0.12}
\pscircle[linecolor=black, linewidth=0.04, dimen=outer](0.604,0.9012127){0.12}
\pscircle[linecolor=black, linewidth=0.04, dimen=outer](1.805,-0.69878733){0.12}
\pscircle[linecolor=black, linewidth=0.04, dimen=outer](5.805,0.9012127){0.12}
\pscircle[linecolor=black, linewidth=0.04, dimen=outer](6.607,2.1012127){0.12}
\psline[linecolor=black, linewidth=0.04](1.404,2.9212127)(0.64069444,1.0012127)
\psline[linecolor=black, linewidth=0.04](1.4206945,2.9012127)(2.5206945,1.7812127)
\psline[linecolor=black, linewidth=0.04](1.404,2.9012127)(6.4606943,2.1212127)
\psline[linecolor=black, linewidth=0.04](4.206,1.7212127)(2.7206945,1.7212127)
\psline[linecolor=black, linewidth=0.04](4.206,1.7012126)(5.6806946,0.9612127)
\psline[linecolor=black, linewidth=0.04](2.5806944,0.5412127)(2.604,1.6012127)
\psline[linecolor=black, linewidth=0.04](2.604,0.5412127)(1.8606944,-0.5987873)
\psline[linecolor=black, linewidth=0.04](2.604,0.4812127)(3.703,-0.6187873)
\psline[linecolor=black, linewidth=0.04](3.404,-2.2787874)(3.7606945,-0.7987873)
\psline[linecolor=black, linewidth=0.04](3.4206944,-2.2587874)(5.6806946,-2.2787874)
\psline[linecolor=black, linewidth=0.04](3.3806944,-2.2987874)(1.904,-2.2987874)
\psline[linecolor=black, linewidth=0.04](0.604,-1.5187873)(1.6806945,-2.2387874)
\psline[linecolor=black, linewidth=0.04](0.604,-1.4787873)(1.7206944,-0.75878733)
\psline[linecolor=black, linewidth=0.04](0.58069444,-1.4587873)(0.58069444,0.8212127)
\psline[linecolor=black, linewidth=0.04](5.805,-0.6787873)(5.7806945,0.80121267)
\psline[linecolor=black, linewidth=0.04](5.805,-0.7187873)(3.904,-0.6787873)
\psline[linecolor=black, linewidth=0.04](5.805,-0.7187873)(5.805,-2.1987872)
\psline[linecolor=black, linewidth=0.04](7.3806944,-1.4787873)(5.904,-2.2787874)
\psline[linecolor=black, linewidth=0.04](7.3806944,-1.4387873)(6.6606946,2.0212126)
\pscircle[linecolor=black, linewidth=0.04, dimen=outer](1.6206944,1.7812127){0.1}
\pscircle[linecolor=black, linewidth=0.04, dimen=outer](2.404,2.2012126){0.1}
\pscircle[linecolor=black, linewidth=0.04, dimen=outer](2.4406943,2.5412128){0.1}
\pscircle[linecolor=black, linewidth=0.04, dimen=outer](4.2406945,2.1212127){0.1}
\pscircle[linecolor=black, linewidth=0.04, dimen=outer](3.604,1.3412127){0.1}
\pscircle[linecolor=black, linewidth=0.04, dimen=outer](4.0,1.0612127){0.1}
\pscircle[linecolor=black, linewidth=0.04, dimen=outer](4.4206944,1.1212127){0.1}
\pscircle[linecolor=black, linewidth=0.04, dimen=outer](2.3606944,-0.23878731){0.1}
\pscircle[linecolor=black, linewidth=0.04, dimen=outer](2.6206944,-0.2587873){0.1}
\pscircle[linecolor=black, linewidth=0.04, dimen=outer](2.904,-0.2587873){0.1}
\pscircle[linecolor=black, linewidth=0.04, dimen=outer](0.7806944,-0.7787873){0.1}
\pscircle[linecolor=black, linewidth=0.04, dimen=outer](1.0806944,-0.9387873){0.1}
\pscircle[linecolor=black, linewidth=0.04, dimen=outer](0.10425,-2.0587873){0.1}
\pscircle[linecolor=black, linewidth=0.04, dimen=outer](3.3606944,-2.9187872){0.1}
\pscircle[linecolor=black, linewidth=0.04, dimen=outer](3.7206945,-1.7587873){0.1}
\pscircle[linecolor=black, linewidth=0.04, dimen=outer](3.8606944,-2.0787873){0.1}
\pscircle[linecolor=black, linewidth=0.04, dimen=outer](5.206,-1.0187873){0.1}
\pscircle[linecolor=black, linewidth=0.04, dimen=outer](5.5206943,-1.3387873){0.1}
\pscircle[linecolor=black, linewidth=0.04, dimen=outer](6.9806943,-0.7387873){0.1}
\pscircle[linecolor=black, linewidth=0.04, dimen=outer](6.6606946,-0.9987873){0.1}
\pscircle[linecolor=black, linewidth=0.04, dimen=outer](6.6206946,-1.3987873){0.1}
\pscircle[linecolor=black, linewidth=0.04, dimen=outer](6.6606946,-1.6587873){0.1}
\psline[linecolor=black, linewidth=0.04](5.8206944,-0.6787873)(6.560694,2.0212126)
\psline[linecolor=black, linewidth=0.04](1.404,2.9012127)(1.5806944,1.8812127)
\psline[linecolor=black, linewidth=0.04](1.4406945,2.9012127)(2.3406944,2.5812128)
\psline[linecolor=black, linewidth=0.04](1.404,2.9012127)(2.3206944,2.2612126)
\psline[linecolor=black, linewidth=0.04](4.206,1.7412127)(4.2206945,2.0412128)
\psline[linecolor=black, linewidth=0.04](4.1606946,1.7212127)(3.6606944,1.4012127)
\psline[linecolor=black, linewidth=0.04](4.1806946,1.7012126)(4.040694,1.1412127)
\psline[linecolor=black, linewidth=0.04](4.206,1.7012126)(4.3806944,1.2212127)
\psline[linecolor=black, linewidth=0.04](7.4206944,-1.4787873)(7.040694,-0.7987873)
\psline[linecolor=black, linewidth=0.04](7.3606944,-1.4787873)(6.7406945,-1.0187873)
\psline[linecolor=black, linewidth=0.04](7.3406944,-1.4987873)(6.706,-1.3787873)
\psline[linecolor=black, linewidth=0.04](7.3606944,-1.4787873)(6.7406945,-1.6387873)
\psline[linecolor=black, linewidth=0.04](5.7606945,-0.7187873)(5.2806945,-0.9787873)
\psline[linecolor=black, linewidth=0.04](5.7806945,-0.6787873)(5.560694,-1.2587873)
\psline[linecolor=black, linewidth=0.04](3.4206944,-2.2987874)(3.7806945,-2.1187873)
\psline[linecolor=black, linewidth=0.04](3.4206944,-2.2387874)(3.6606944,-1.8187873)
\psline[linecolor=black, linewidth=0.04](2.5606945,0.5412127)(2.3606944,-0.15878731)
\psline[linecolor=black, linewidth=0.04](2.5806944,0.4612127)(2.6206944,-0.15878731)
\psline[linecolor=black, linewidth=0.04](2.6206944,0.5012127)(2.8606944,-0.15878731)
\psline[linecolor=black, linewidth=0.04](0.56069446,-1.4587873)(0.16069442,-1.9787873)
\psline[linecolor=black, linewidth=0.04](3.3806944,-2.2987874)(3.3606944,-2.8187873)
\psline[linecolor=black, linewidth=0.04](0.58069444,-1.4787873)(1.0206944,-0.9987873)
\psline[linecolor=black, linewidth=0.04](0.604,-1.4587873)(0.76069444,-0.8587873)
\end{pspicture}
}
\end{center}
\caption{Example of the bipartite map obtained by erasing the labels of $\Gamma$.}
\label{figrepresent2}
\end{figure}
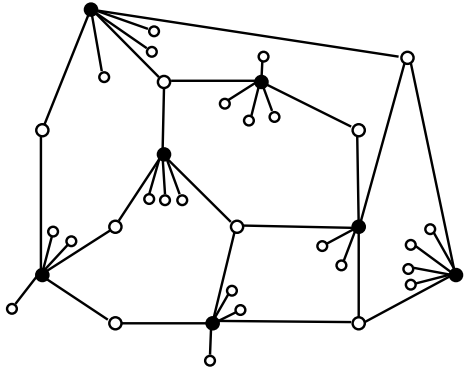

So, the $n$-regular map $\Gamma_r$ admits a labeling which can be transported  to a labeling of  the map $\tilde{\Gamma}$. Moreover, the map $\tilde{\Gamma}$ equipped with this labeling is an increasing bipartite map with equal number of faces and black vertices, each of constant degree. From this, one can construct a branched covering with $\tilde{\Gamma}$ as a representation.

In conclusion, there is a bijection between the set of $n$-regular maps and the set of representations of branched coverings with $n$ critical values up to the cyclic action of the permutation $(1 ~ 2 \cdots n)$ on the label of the representation.

\subsection{Definition}

In the previous subsection, we proved the existence of two different applications (a bijection and a surjection) in order to reduce the complexity of the map $\Gamma$. Obviously, we want to know if we can combine the two ideas used previously. So, let $\Gamma$ be a representation of a branched covering, and denote by $G$ the map obtained by erasing the label and each $1$-valence white vertex with its unique incident edge.

\begin{defrealiz}
We call $G$ a \textbf{skeleton} of the branched covering $\pi$.
\end{defrealiz}

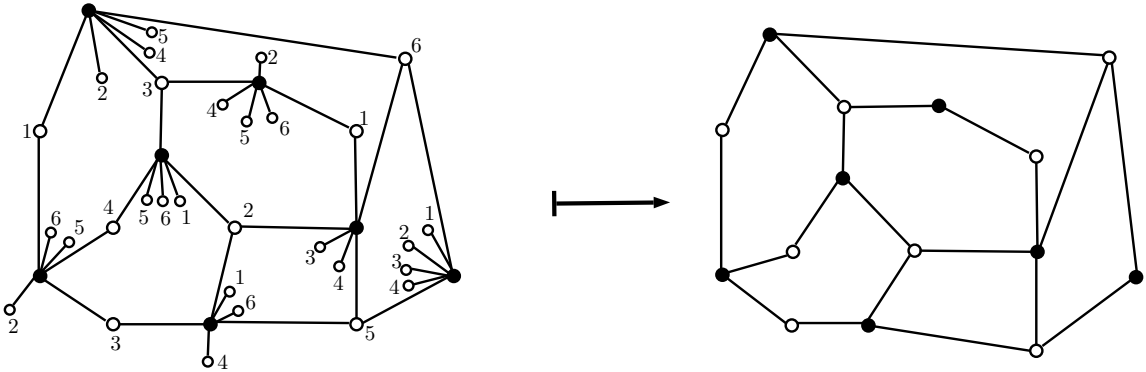
\begin{figure}[ht]
\begin{center}
\psscalebox{0.8 0.8} 
{
\begin{pspicture}(0,-3.064247)(18.738964,3.064247)
\pscircle[linecolor=black, linewidth=0.04, fillstyle=solid,fillcolor=black, dimen=outer](1.404,2.9459777){0.12}
\pscircle[linecolor=black, linewidth=0.04, fillstyle=solid,fillcolor=black, dimen=outer](0.604,-1.4540223){0.12}
\pscircle[linecolor=black, linewidth=0.04, fillstyle=solid,fillcolor=black, dimen=outer](2.604,0.5459777){0.12}
\pscircle[linecolor=black, linewidth=0.04, fillstyle=solid,fillcolor=black, dimen=outer](4.206,1.7459778){0.12}
\pscircle[linecolor=black, linewidth=0.04, fillstyle=solid,fillcolor=black, dimen=outer](5.805,-0.6540223){0.12}
\pscircle[linecolor=black, linewidth=0.04, fillstyle=solid,fillcolor=black, dimen=outer](7.404,-1.4540223){0.12}
\pscircle[linecolor=black, linewidth=0.04, fillstyle=solid,fillcolor=black, dimen=outer](3.404,-2.2540224){0.12}
\pscircle[linecolor=black, linewidth=0.04, dimen=outer](1.805,-2.2540224){0.12}
\pscircle[linecolor=black, linewidth=0.04, dimen=outer](5.805,-2.2540224){0.12}
\pscircle[linecolor=black, linewidth=0.04, dimen=outer](3.805,-0.6540223){0.12}
\pscircle[linecolor=black, linewidth=0.04, dimen=outer](2.604,1.7459778){0.12}
\pscircle[linecolor=black, linewidth=0.04, dimen=outer](0.604,0.9459777){0.12}
\pscircle[linecolor=black, linewidth=0.04, dimen=outer](1.805,-0.6540223){0.12}
\pscircle[linecolor=black, linewidth=0.04, dimen=outer](5.805,0.9459777){0.12}
\pscircle[linecolor=black, linewidth=0.04, dimen=outer](6.607,2.1459777){0.12}
\psline[linecolor=black, linewidth=0.04](1.404,2.9659777)(0.64069444,1.0459777)
\psline[linecolor=black, linewidth=0.04](1.4206945,2.9459777)(2.5206945,1.8259777)
\psline[linecolor=black, linewidth=0.04](1.404,2.9459777)(6.4606943,2.1659777)
\psline[linecolor=black, linewidth=0.04](4.206,1.7659777)(2.7206945,1.7659777)
\psline[linecolor=black, linewidth=0.04](4.206,1.7459778)(5.6806946,1.0059777)
\psline[linecolor=black, linewidth=0.04](2.5806944,0.58597773)(2.604,1.6459777)
\psline[linecolor=black, linewidth=0.04](2.604,0.58597773)(1.8606944,-0.5540223)
\psline[linecolor=black, linewidth=0.04](2.604,0.52597773)(3.703,-0.5740223)
\psline[linecolor=black, linewidth=0.04](3.404,-2.2340224)(3.7606945,-0.7540223)
\psline[linecolor=black, linewidth=0.04](3.4206944,-2.2140224)(5.6806946,-2.2340224)
\psline[linecolor=black, linewidth=0.04](3.3806944,-2.2540224)(1.904,-2.2540224)
\psline[linecolor=black, linewidth=0.04](0.604,-1.4740223)(1.6806945,-2.1940222)
\psline[linecolor=black, linewidth=0.04](0.604,-1.4340223)(1.7206944,-0.7140223)
\psline[linecolor=black, linewidth=0.04](0.58069444,-1.4140223)(0.58069444,0.8659777)
\psline[linecolor=black, linewidth=0.04](5.805,-0.6340223)(5.7806945,0.8459777)
\psline[linecolor=black, linewidth=0.04](5.805,-0.6740223)(3.904,-0.6340223)
\psline[linecolor=black, linewidth=0.04](5.805,-0.6740223)(5.805,-2.1540222)
\psline[linecolor=black, linewidth=0.04](7.3806944,-1.4340223)(5.904,-2.2340224)
\psline[linecolor=black, linewidth=0.04](7.3806944,-1.3940223)(6.6606946,2.0659778)
\pscircle[linecolor=black, linewidth=0.04, dimen=outer](1.6206944,1.8259777){0.1}
\pscircle[linecolor=black, linewidth=0.04, dimen=outer](2.404,2.2459776){0.1}
\pscircle[linecolor=black, linewidth=0.04, dimen=outer](2.4406943,2.5859778){0.1}
\pscircle[linecolor=black, linewidth=0.04, dimen=outer](4.2406945,2.1659777){0.1}
\pscircle[linecolor=black, linewidth=0.04, dimen=outer](3.604,1.3859777){0.1}
\pscircle[linecolor=black, linewidth=0.04, dimen=outer](4.0,1.1059777){0.1}
\pscircle[linecolor=black, linewidth=0.04, dimen=outer](4.4206944,1.1659777){0.1}
\pscircle[linecolor=black, linewidth=0.04, dimen=outer](2.3606944,-0.1940223){0.1}
\pscircle[linecolor=black, linewidth=0.04, dimen=outer](2.6206944,-0.2140223){0.1}
\pscircle[linecolor=black, linewidth=0.04, dimen=outer](2.904,-0.2140223){0.1}
\pscircle[linecolor=black, linewidth=0.04, dimen=outer](0.7806944,-0.7340223){0.1}
\pscircle[linecolor=black, linewidth=0.04, dimen=outer](1.0806944,-0.8940223){0.1}
\pscircle[linecolor=black, linewidth=0.04, dimen=outer](0.10425,-2.0140224){0.1}
\pscircle[linecolor=black, linewidth=0.04, dimen=outer](3.3606944,-2.8740222){0.1}
\pscircle[linecolor=black, linewidth=0.04, dimen=outer](3.7206945,-1.7140223){0.1}
\pscircle[linecolor=black, linewidth=0.04, dimen=outer](3.8606944,-2.0340223){0.1}
\pscircle[linecolor=black, linewidth=0.04, dimen=outer](5.206,-0.97402227){0.1}
\pscircle[linecolor=black, linewidth=0.04, dimen=outer](5.5206943,-1.2940223){0.1}
\pscircle[linecolor=black, linewidth=0.04, dimen=outer](6.9806943,-0.6940223){0.1}
\pscircle[linecolor=black, linewidth=0.04, dimen=outer](6.6606946,-0.9540223){0.1}
\pscircle[linecolor=black, linewidth=0.04, dimen=outer](6.6206946,-1.3540223){0.1}
\pscircle[linecolor=black, linewidth=0.04, dimen=outer](6.6606946,-1.6140223){0.1}
\psline[linecolor=black, linewidth=0.04](5.8206944,-0.6340223)(6.560694,2.0659778)
\psline[linecolor=black, linewidth=0.07, tbarsize=0.0706cm 5.0,arrowsize=0.053cm 2.0,arrowlength=1.4,arrowinset=0.0]{|->}(9.020695,-0.2540223)(10.960694,-0.23402229)
\psline[linecolor=black, linewidth=0.04](1.404,2.9459777)(1.5806944,1.9259777)
\psline[linecolor=black, linewidth=0.04](1.4406945,2.9459777)(2.3406944,2.6259778)
\psline[linecolor=black, linewidth=0.04](1.404,2.9459777)(2.3206944,2.3059778)
\psline[linecolor=black, linewidth=0.04](4.206,1.7859777)(4.2206945,2.0859778)
\psline[linecolor=black, linewidth=0.04](4.1606946,1.7659777)(3.6606944,1.4459777)
\psline[linecolor=black, linewidth=0.04](4.1806946,1.7459778)(4.040694,1.1859777)
\psline[linecolor=black, linewidth=0.04](4.206,1.7459778)(4.3806944,1.2659777)
\psline[linecolor=black, linewidth=0.04](7.4206944,-1.4340223)(7.040694,-0.7540223)
\psline[linecolor=black, linewidth=0.04](7.3606944,-1.4340223)(6.7406945,-0.97402227)
\psline[linecolor=black, linewidth=0.04](7.3406944,-1.4540223)(6.706,-1.3340223)
\psline[linecolor=black, linewidth=0.04](7.3606944,-1.4340223)(6.7406945,-1.5940223)
\psline[linecolor=black, linewidth=0.04](5.7606945,-0.6740223)(5.2806945,-0.9340223)
\psline[linecolor=black, linewidth=0.04](5.7806945,-0.6340223)(5.560694,-1.2140223)
\psline[linecolor=black, linewidth=0.04](3.4206944,-2.2540224)(3.7806945,-2.0740223)
\psline[linecolor=black, linewidth=0.04](3.4206944,-2.1940222)(3.6606944,-1.7740223)
\psline[linecolor=black, linewidth=0.04](2.5606945,0.58597773)(2.3606944,-0.11402229)
\psline[linecolor=black, linewidth=0.04](2.5806944,0.5059777)(2.6206944,-0.11402229)
\psline[linecolor=black, linewidth=0.04](2.6206944,0.5459777)(2.8606944,-0.11402229)
\psline[linecolor=black, linewidth=0.04](0.56069446,-1.4140223)(0.16069442,-1.9340223)
\psline[linecolor=black, linewidth=0.04](3.3806944,-2.2540224)(3.3606944,-2.7740223)
\psline[linecolor=black, linewidth=0.04](0.58069444,-1.4340223)(1.0206944,-0.9540223)
\psline[linecolor=black, linewidth=0.04](0.604,-1.4140223)(0.76069444,-0.8140223)
\rput[bl](2.2806945,1.5259777){$3$}
\rput[bl](0.3044,0.8259777){$1$}
\rput[bl](1.7606944,-2.6740222){$3$}
\rput[bl](5.9406943,-2.5140224){$5$}
\rput[bl](5.8406944,1.0459777){$1$}
\rput[bl](6.706,2.2659776){$6$}
\rput[bl](3.9406943,-0.4740223){$2$}
\rput[bl](1.6406944,-0.4340223){$4$}
\rput[bl](4.3406944,2.0659778){$2$}
\rput[bl](3.3406944,1.1859777){$4$}
\rput[bl](3.8806944,0.7259777){$5$}
\rput[bl](4.540694,0.8859777){$6$}
\rput[bl](1.5406945,1.4659777){$2$}
\rput[bl](2.505,2.0859778){$4$}
\rput[bl](2.5406945,2.4059777){$5$}
\rput[bl](0.08069443,-2.3940222){$2$}
\rput[bl](1.1606945,-0.7540223){$5$}
\rput[bl](0.7806944,-0.6140223){$6$}
\rput[bl](2.2206945,-0.5940223){$5$}
\rput[bl](2.5806944,-0.6140223){$6$}
\rput[bl](2.9206944,-0.5940223){$1$}
\rput[bl](3.5206945,-2.9940224){$4$}
\rput[bl](3.9806945,-1.9940223){$6$}
\rput[bl](3.805,-1.5940223){$1$}
\rput[bl](4.9606943,-1.2540222){$3$}
\rput[bl](5.4206944,-1.6940223){$4$}
\rput[bl](6.9206944,-0.5340223){$1$}
\rput[bl](6.503,-0.8340223){$2$}
\rput[bl](6.3806944,-1.3140223){$3$}
\rput[bl](6.3406944,-1.7540222){$4$}
\pscircle[linecolor=black, linewidth=0.04, fillstyle=solid,fillcolor=black, dimen=outer](12.60,2.5459776){0.12}
\pscircle[linecolor=black, linewidth=0.04, fillstyle=solid,fillcolor=black, dimen=outer](15.380694,1.3659778){0.12}
\pscircle[linecolor=black, linewidth=0.04, fillstyle=solid,fillcolor=black, dimen=outer](11.820694,-1.4340223){0.12}
\pscircle[linecolor=black, linewidth=0.04, fillstyle=solid,fillcolor=black, dimen=outer](14.220695,-2.2740223){0.12}
\pscircle[linecolor=black, linewidth=0.04, fillstyle=solid,fillcolor=black, dimen=outer](13.80,0.1659777){0.12}
\pscircle[linecolor=black, linewidth=0.04, fillstyle=solid,fillcolor=black, dimen=outer](17.00,-1.0540223){0.12}
\pscircle[linecolor=black, linewidth=0.04, fillstyle=solid,fillcolor=black, dimen=outer](18.620695,-1.4740223){0.12}
\pscircle[linecolor=black, linewidth=0.04, dimen=outer](11.820694,0.9659777){0.12}
\pscircle[linecolor=black, linewidth=0.04, dimen=outer](13.820694,1.3459777){0.12}
\pscircle[linecolor=black, linewidth=0.04, dimen=outer](12.980695,-1.0540223){0.12}
\pscircle[linecolor=black, linewidth=0.04, dimen=outer](14.980695,-1.0340223){0.12}
\pscircle[linecolor=black, linewidth=0.04, dimen=outer](12.960694,-2.2740223){0.12}
\pscircle[linecolor=black, linewidth=0.04, dimen=outer](16.980694,-2.6940222){0.12}
\pscircle[linecolor=black, linewidth=0.04, dimen=outer](16.980694,0.52597773){0.12}
\pscircle[linecolor=black, linewidth=0.04, dimen=outer](18.180695,2.1659777){0.12}
\psline[linecolor=black, linewidth=0.04](12.580694,2.5659778)(11.880694,1.0659777)
\psline[linecolor=black, linewidth=0.04](12.640695,2.5859778)(13.740694,1.4459777)
\psline[linecolor=black, linewidth=0.04](12.580694,2.5659778)(18.080694,2.2059777)
\psline[linecolor=black, linewidth=0.04](15.340694,1.3659778)(13.920694,1.3459777)
\psline[linecolor=black, linewidth=0.04](15.40,1.3659778)(16.860695,0.58597773)
\psline[linecolor=black, linewidth=0.04](17.00,-1.0340223)(16.980694,0.4259777)
\psline[linecolor=black, linewidth=0.04](17.020695,-1.0540223)(15.060695,-1.0340223)
\psline[linecolor=black, linewidth=0.04](17.020695,-1.0340223)(18.160694,2.0859778)
\psline[linecolor=black, linewidth=0.04](18.640694,-1.4540223)(18.220694,2.0859778)
\psline[linecolor=black, linewidth=0.04](18.60,-1.4740223)(17.080694,-2.6140223)
\psline[linecolor=black, linewidth=0.04](16.980694,-1.0540223)(16.980694,-2.5940223)
\psline[linecolor=black, linewidth=0.04](14.220695,-2.2740223)(16.880695,-2.6940222)
\psline[linecolor=black, linewidth=0.04](14.180695,-2.2340224)(14.90,-1.1140223)
\psline[linecolor=black, linewidth=0.04](13.80,0.20597771)(14.90,-0.9540223)
\psline[linecolor=black, linewidth=0.04](13.80,0.20597771)(13.820694,1.2659777)
\psline[linecolor=black, linewidth=0.04](13.780694,0.20597771)(13.020695,-0.9340223)
\psline[linecolor=black, linewidth=0.04](14.20,-2.2340224)(13.060695,-2.2340224)
\psline[linecolor=black, linewidth=0.04](11.80,-1.4140223)(12.860695,-2.2340224)
\psline[linecolor=black, linewidth=0.04](11.840694,-1.4140223)(12.880694,-1.1140223)
\psline[linecolor=black, linewidth=0.04](11.80,-1.4140223)(11.80,0.8859777)
\end{pspicture}
}
\end{center}
\caption{Example of a skeleton of a branched covering $\pi$ of degree $7$.}
\label{figconsreal}
\end{figure}

Notice that the function that sends a representation $\Gamma$ to a skeleton $G$ is not injective, in the sense that we can construct the representations of two different branched coverings that are sent to the same skeleton. For example, we can relabel by $7$ one of the white vertices initially labeled by $1$ in Figure \ref{figrepresent1}. This changes the representation $\Gamma$ and the associated branched coverings, but not the skeleton $G$.

By construction, $G$ is  a bipartite map, and so now a natural question is: which bipartite map comes from a skeleton of a branched covering? To answer this question, motivated by W. Thurston's treatment of the generic case, we introduce the following notion for bipartite maps: 

\begin{defbalmap}
\label{defbalmap}
A bipartite map $G$ is called a \textbf{balanced map}, or we just say that $G$ is balanced, if it satisfies the  following two conditions.
\begin{itemize}
\item (\textbf{global balance condition}) $G$ has as many faces as black vertices. We say that $G$ is globally balanced.
\item (\textbf{local balance condition}) for each submap $H$ containing at least one black vertex, the number of black vertices of $H$ is greater than or equal to the number of faces of $H$, i.e. $V_H \geq F_H$. We say that $G$ is locally balanced.
\end{itemize}
\end{defbalmap}

A star-like bipartite map with a central black vertex is a balanced map. A map with one vertex of each color and two edges (therefore two faces) is not balanced. 
A quadrilateral (i.e. a topological model of a square) is the simplest example of a balanced map that is not a tree.  
Some further examples of balanced maps can be found in Figures \ref{figbipmap} and \ref{figconsreal}.
  
By construction of a skeleton $G$, the global balance condition is natural (as this condition is still satisfied by the map $\Gamma$), but not sufficient to define a skeleton. We prove later that this condition is equivalent to the Riemann-Hurwitz condition (cf Proposition \ref{propRHcond}). This correlation between the two conditions allow us to make a connection between our model and the passport of a branched covering.

Notice that by definition, if $G$ is a bipartite map that satisfies the local balance condition, and $H \subset G$ with at least one black vertex, then $H$ also satisfies the local balance condition. 

\begin{lemconvsubmap}
\label{lemconvsubmap}
Let $G$ be a bipartite map, and suppose that $G$ is globally balanced. If for each full submap $H$ with at least one black vertex, we have the inequality $V_H \geq F_H$ then $G$ is balanced.
\end{lemconvsubmap}

In other words, we just need to verify the local condition for each full submap of $G$. Depending on the situation, we prefer to use Definition \ref{defbalmap} rather than Lemma \ref{lemconvsubmap}.

\begin{proof}
Let $H$ be a non full submap of $G$ with at least one black vertex. We want to prove that $F_H \leq V_H$. First, by Proposition \ref{propconnconv}, there exists a finite number $n$ of full submaps of $G$, denoted by $H_1,\ldots, H_n$, such that 
\begin{displaymath}
\bigcap_{i}{H_i} = H\quad\text{and}\quad  H_i\cup H_j=G\quad \text{for any pair } i\ne j.
\end{displaymath}
 Then, by hypothesis,
\begin{align*}
F_G & = V_G \\
F_{H_i} & \leq V_{H_i}, \quad \text{for } i=1,\ldots ,n.
\end{align*}
So, using  Property \ref{propdegunin}, we deduce that for each pair $(i,j) \in \{1,\ldots,n\}$, $i \neq j$, 
\begin{align*}
F_{H_i \cap H_j} & = F_{H_i} + F_{H_j} - F_{G} \\
& \leq V_{H_i} + V_{H_j} - V_{G} \\
& = V_{H_i \cap H_j}.
\end{align*}
By induction, we prove similarly that the intersection of any number of maps $H_i$ satisfies the same relation. In particular, we have
\begin{displaymath}
F_H = F_{\bigcap H_i} \leq V_{\bigcap H_i} = V_H.
\end{displaymath} 

Suppose now that $H$ has $k$ connected components $L_1,\ldots, L_k$. If one of them does not contain black vertices, then it does not contain  any edges so has a single white vertex. Erasing this connected component with not change the number of black vertices nor that of faces. Erasing now all the single-white-vertex connected components, each remaining connected component contains a black vertex (and there is at least one such component).
So we may as well assume that
each $L_i$, $i=1, \ldots , k$, contains at least one black vertex.

Then, using the previous result, we have
\begin{displaymath}
V_{L_i} \geq F_{L_i} \quad \forall i=1, \ldots , k.
\end{displaymath}
By Property \ref{propdegunin}, we deduce that
\begin{displaymath}
V_H = \sum_{i=1}^{k}{V_{L_i}} \geq \sum_{i=1}^{k}{F_{L_i}} \geq \sum_{i=1}^{k}{F_{L_i}} - k + 1 = F_H.
\end{displaymath}
\end{proof}

We end this subsection by a last proposition on balanced maps.

\begin{corconv}
\label{corconv}
Let $G$ be a balanced map, and $\mathcal{F}$ be a subset of $\mathcal{F}_G$, $\mathcal{F} \neq \mathcal{F}_G$. Denote by $\mathcal{V}$ the subset of $\mathcal{V}_G$ that contains every black vertex having an incidence with a face of $\mathcal{F}$, and by $F$, resp. $V$, the cardinal of $\mathcal{F}$, resp. $\mathcal{V}$. Then $F < V$.
\end{corconv}

\begin{proof}
By erasing the vertices and edges of $G$ that are not incident to a face of $\mathcal{F}$, we construct a submap (or a disjoint union of submap) of $G$. Without loss of generality, we can suppose that this construction gives a submap $H$ of $G$. By construction, the number of vertices of $H$ is equal to $V$ and the number of faces of $H$ is equal to $F+1$. As $G$ is locally balanced, we deduce immediately by Definition \ref{defbalmap} that $F+1 \leq V$.
\end{proof}

\subsection{From balanced maps to balanced graphs}

The goal of this subsection is to give some basic results about balanced maps in order to manipulate this notion.

\begin{propwhitedeg}
\label{propwhitedeg}
Let $G$ be a balanced map. Then a white vertex $w$ of $G$ can not be doubly incident to any face of $G$. In particular, no edge of $G$ is doubly incident to a face, except if the white tip is of degree $1$.

Moreover adding or erasing some $1$-degree white vertices together with its attached edge changes neither the global nor the local balance condition.
\end{propwhitedeg}

\begin{proof}
Let us argue by contradiction. Suppose that there is a Jordan arc inside a face forming a Jordan curve $\gamma$ with a white vertex $w$, with $\gamma$ separating the black vertices of $G$.
Erasing the vertices and edges of $G$ in each complementary component of $\gamma$ gives  two  submaps $G_1$ and $G_2$. We have $G=G_1\cup G_2$,  and  $G_1\cap G_2$ is reduced to a single white vertex. As $G_1$ and $G_2$ are  submaps of $G$ containing black vertices, and $G$ is locally balanced, we deduce that
\begin{equation}
\label{whitedeg1}
F_{G_i} \leq V_{G_i}, \quad \text{for } i \in \{1,2\}.
\end{equation}
Thus, by construction,  
\begin{align*}
F_G & = F_{G_1} + F_{G_2} - 1\le V_{G_1} + V_{G_2} -1 = V_G-1< V_G. 
\end{align*}
This contradicts the fact that $G$ is globally balanced.

The rest of the proposition is obvious.
\end{proof}

So, a bipartite map that satisfies both global and local conditions has some constraints. The next result gives another point of view about the balanced conditions.

\begin{thbalmaps}
\label{thbalmaps}
Let $G$ be a bipartite map. Then $G$ is globally balanced if and only if
\begin{displaymath}
\sum_{w \in  \mathcal{W}_G}{\left( \deg_G(w)-1 \right)} = 2V_G -2.
\end{displaymath}
Similarly, $G$ is  locally balanced if and only if for each submap $H$ of $G$, we have
\begin{displaymath}
\sum_{w \in  \mathcal{W}_H}{\left( \deg_H(w)-1 \right)} \leq 2V_H -2.
\end{displaymath}
\end{thbalmaps}

The reader may notice some similarity with the Riemann-Hurwitz condition. More exactly, using Theorem \ref{thbalmap}, we can prove that there exists a relation between the Riemann-Hurwitz condition for branched data and the global balance condition for balanced maps.

\begin{proof}
Let $G$ be a bipartite map. By Euler characteristic, we have
\begin{displaymath}
F_G + (V_G + W_G) = E_G + 2,
\end{displaymath}
thus
\begin{equation}
\label{balmaps2}
E_G - W_G = F_G + V_G - 2.
\end{equation}
Now, as each edge of $G$ belongs to one and only one white vertex, the definition of the degree of a vertex implies that
\begin{displaymath}
\sum_{w \in \mathcal{W}_G}{\deg_G(w)} = E_G.
\end{displaymath}
From this relation and equality (\ref{balmaps2}), we deduce that
\begin{equation}
\label{balmaps1}
\sum_{w \in \mathcal{W}_G}{\left( \deg_G(v)-1 \right)} = E_G - W_G =F_G + V_G - 2.
\end{equation}
Therefore \begin{equation}
\label{balmaps3}
F_G=V_G \Longleftrightarrow \sum_{w \in \mathcal{W}_G}{\left( \deg_G(v)-1 \right)} =  2 V_G - 2
\end{equation}
and \begin{equation}
\label{balmaps4}
F_G\le V_G \Longleftrightarrow \sum_{w \in \mathcal{W}_G}{\left( \deg_G(v)-1 \right)} \le 2 V_G - 2
\end{equation}

Applying (\ref{balmaps3}) to $G$, we can still deduce the required condition such that $G$ is globally balanced.

\vspace{2mm}

Now, assume that  $G$ is locally balanced, and let $H$ be a submap of $ G$. Then by definition, we have the relation $F_H \le V_H $. Applying (\ref{balmaps4}) to  $H$ we get 
$$\sum_{w \in \mathcal{W}_H}{\left( \deg_H(v)-1 \right)} \le 2 V_H - 2.$$

Conversely assume that the above inequality holds for every submap $H$ of $G$. Then applying (\ref{balmaps4}) to $H$, we deduce immediately that $G$ is locally balanced.
\end{proof}

This result is important as it implies that the concept of balanced map may be defined using only the vertices and the edges of the map. In particular, it implies the following result.

\begin{corbalmaps}
\label{corbalmaps}
Let $G_1$ and $G_2$ be two bipartite maps, and suppose that $G_1$ and $G_2$ are equivalent as graphs. Then $G_1$ is globally (resp. locally) balanced if and only if $G_2$ is globally (resp. locally) balanced.
\end{corbalmaps}

\subsection{Generalization of Thurston's result}

In this subsection, we finally answer our initial question: determine the bipartite maps that are skeletons of  branched coverings. 

\begin{thbalmap}
\label{thbalmap}
A bipartite map $G$ is a skeleton of a branched covering if and only if $G$ is a balanced map.
\end{thbalmap}

\begin{proof}
To prove this Theorem, we need to come back to the construction of $G$. In fact, if $G$ is a skeleton of a branched covering, then there exists a representation $\Gamma$, i.e. an increasing bipartite map with as many black vertices as faces each of the same degree, such that $G = \phi(\Gamma)$. Moreover, we prove in subsection \ref{construction} that a representative is essentially defined by a regular planar map, or equivalently by a bipartite map $\tilde{\Gamma}$ with as many black vertices as faces, each of the same degree (we simply erase the labeling of the map $\Gamma$). In conclusion, if $G$ is a skeleton of a branched covering, then by adding some $1$-degree white vertices and edges to $G$, we can construct a map $\tilde{\Gamma}$. Conversely, if such a construction is possible, then $G$ is a skeleton of a branched covering.

\vspace{2mm}

First, suppose that $G$ is a balanced map with $n$ white vertices, and add inside each face $f$ of $G$ a number of white dots such that the number of these dots plus the number of white vertex having an incidence with $f$ is equal to $n$. Notice that by Proposition \ref{propwhitedeg}, the number of white vertex having an incidence with a given face of $G$ is equal to the degree of this face. Now, we want to prove that we can connect each of these added dots to a black vertex having an incidence with the same face $f$ by an edge in order to obtain a map where each black vertex is of degree $n$, and so the construction mentioned earlier is possible.

In order to prove this result, consider $\mathcal{A}$ the set containing all the black vertices $v$ of $G$, each repeated $(n-\deg_G(v))$ times, and denote by $\mathcal{S}$ a family of finite subsets $\{A_1, A_2, \ldots \}$ of $\mathcal{A}$ where each subset $A_i$ corresponds to a white dot $w$ and contains all the black vertices $v$ of $G$ (counted $(n-\deg_G(v))$ times) that can be connected to $w$. Then our construction is possible if and only if $\mathcal{S}$ is a perfect matching. So, let's prove that the family $\mathcal{S}$ is a perfect matching. 

Denote by $A$ the cardinal of $\mathcal{A}$, and by $S$ the cardinal of $\mathcal{S}$. By definition, we have
\begin{displaymath}
A = \sum_{v \in \mathcal{V}_G}{(n-\deg_G(v))} = n.V_G - \sum_{v \in \mathcal{V}_G}{\deg_G(v)},
\end{displaymath}
and
\begin{displaymath}
S = \sum_{f \in \mathcal{F}_G}{(n-\deg_G(f))} = n.F_G - \sum_{f \in \mathcal{F}_G}{\deg_G(f)}.
\end{displaymath}
Moreover, as $G$ is a bipartite map, then
\begin{displaymath}
\sum_{f \in \mathcal{F}_G}{\deg_G(f)} = \sum_{v \in \mathcal{V}_G}{\deg_G(v)} = E_G,
\end{displaymath}
as each edge of $G$ is incident to a unique black vertex. In conclusion, we have the relations
\begin{align}
\label{eqbalmap1}
A & = n.V_G - E_G, \\
\label{eqbalmap2}
S & = n.F_G - E_G.
\end{align}
As $G$ is a balanced map, then $V_G = F_G$, and so using relations (\ref{eqbalmap1}) and (\ref{eqbalmap2}) we can conclude that $A = S$. Now, we just have to prove that $\mathcal{S}$ satisfies the marriage condition. 

Consider $\mathcal{S}'$ a subfamily of $\mathcal{S}$, and denote by $\mathcal{A}'$ the subset of $\mathcal{A}$ containing each element of $\mathcal{A}$ that belongs to at least one subset of $\mathcal{S}'$. Finally, denote by $S'$, resp. $A'$, the cardinal of $\mathcal{S}'$, resp. $\mathcal{A}'$. Then, we have to prove that $S' \leq A'$. Up to adding some subsets of $\mathcal{S}$ into $\mathcal{S}'$, we can suppose that if a subset $A_i$ corresponding to a white dot $w$ belongs to $\mathcal{S}'$ then each other subset corresponding to a white dot in the same face of $G$ as $w$ also belongs to $\mathcal{S}'$, as this operation does not increase the value $A'$. 

Then, there exists a subset $\mathcal{F}$ of $\mathcal{F}_G$ such that the family $\mathcal{S}'$ corresponds to all the white dots belonging to a face of $\mathcal{F}$. We denote by $\mathcal{V}$ the subset of $\mathcal{V}_G$ that contains each black vertex having an incidence with a face of $\mathcal{F}$. Now consider the submap $H$ of $G$ obtained by erasing each edge of $G$ that is not incident to a vertex of $\mathcal{V}$, as well as each vertex that is not incident to one of the remaining edges. Clearly, by construction, we have
\begin{displaymath}
\mathcal{V} = \mathcal{V}_H \quad \text{and} \quad \mathcal{F} \subset \mathcal{F}_H.
\end{displaymath}
Let's prove our inequality. By definition, we have
\begin{equation}
\label{eqaprim}
A' = \sum_{v \in \mathcal{V}}{(n - \deg_G(v))} = \sum_{v \in \mathcal{V}_H}{(n - \deg_H(v))} = n.V_H - \sum_{v \in \mathcal{V}_H}{\deg_H(v)} = n.V_H - E_H.
\end{equation} 
So, we can deduce that
\begin{align*}
S' & = \sum_{f \in \mathcal{F}}{(n - \deg_G(f))} \\
& = \sum_{f \in \mathcal{F}}{(n - \deg_H(f))} \\
& \leq \sum_{f \in \mathcal{F}_H}{(n - \deg_H(f))} \\
& = n.F_H - E_H \\
& \leq n.V_H - E_H, \quad \text{as } G \text{ is locally balanced}.
\end{align*}
In conclusion, using the relation (\ref{eqaprim}), we have $S' \leq A'$ and so $\mathcal{S}$ satisfies the marriage condition.

\vspace{2mm}

Conversely, suppose that $G$ is a skeleton of a branched covering, then by adding some $1$-degree white vertices and edges to $G$, we can construct a bipartite map $\tilde{\Gamma}$ such that each black vertex and face have the same degree $n$. As previously, consider $\mathcal{A}$ the set containing all the black vertices $v$ of $G$, each repeated $(n-\deg_G(v))$ times, and denote by $\mathcal{S}$ a family of finite subsets $\{A_1, A_2, \ldots \}$ of $\mathcal{A}$ where each subset $A_i$ corresponds to a $1$-degree white vertex $w$ of $\tilde{\Gamma}$ and contains all the black vertices $v$ of $G$ (counted $(n-\deg_G(v))$ times) being incident to the same face as $w$. The existence of the map $\tilde{\Gamma}$ implies that $\mathcal{S}$ is a perfect matching. 

Denote by $S$, resp. $A$, the cardinal of $\mathcal{S}$, resp. $\mathcal{A}$. By definition, $S$ and $A$ satisfy the relations (\ref{eqbalmap1}) and (\ref{eqbalmap2}), and as $\mathcal{S}$ is a perfect matching, $S = A$. So we can still deduce that $F_G = V_G$, i.e. $G$ is globally balanced. 

Now, consider a full submap $H$ of $G$ with at least two faces, $H \neq G$, and denote by $f_H$ the unique face of $H$ that is not a face of $G$. Then, let $\mathcal{S}'$ be the subfamily of $\mathcal{S}$ that corresponds to all the $1$-degree white vertices of $G$ belonging to a face of $\mathcal{F}_H \setminus \{f_H\}$, and $\mathcal{A}'$ be the subset of $\mathcal{A}$ containing each element of $\mathcal{A}$ that belongs to at least one subset of $\mathcal{S}'$. Then, we prove easily that
\begin{align*}
S' := card(\mathcal{S}') & = \sum_{f \in \mathcal{F}_H \setminus \{f_H\}}{(n-\deg_G(f))}, \\
& = \sum_{f \in \mathcal{F}_H \setminus \{f_H\}}{(n-\deg_H(f))}.
\end{align*}
Thus, \begin{equation}
\label{eqbalmap3}
S' = n.F_H - A_H - (n-\deg_H(f_H)).
\end{equation} Similarly, we have
\begin{align*}
A' := card(\mathcal{A}') & = \sum_{v \in \mathcal{V}_H}{(n-\deg_G(v))} \\
& \leq \sum_{v \in \mathcal{V}_H}{(n-\deg_H(v))},
\end{align*}
and so \begin{equation}
\label{eqbalmap4}
A' \leq n \cdot V_H - A_H.
\end{equation}

Moreover, as $\mathcal{S}$ satisfies the marriage condition, then $S' - A' \leq 0$, and so using relations (\ref{eqbalmap3}) and (\ref{eqbalmap4}), we have 
\begin{displaymath}
0 \geq S' - A' \geq n \cdot (F_H - V_H) - (n-\deg_H(f_H)).
\end{displaymath}
Finally, we deduce that
\begin{displaymath}
F_H - V_H \leq 1 - \frac{1}{n} \deg_H(f_H) < 1.
\end{displaymath}
In conclusion $F_H \leq V_H$, and so $G$ is locally balanced.
\end{proof}

A direct consequence of this geometric interpretation of a branched covering of the sphere, and of the Proposition \ref{propwhitedeg} is the following result.

\begin{corbalmap}
Let $\pi$ be a branched covering of the sphere $\mathbb{S}^2$, and $U$ a simply connected subset of $\mathcal{S}^2$ containing at most one critical value of $\pi$. Then the pull-back of $\mathbb{S}^2 \setminus U$ by $\pi$ is connected.
\end{corbalmap}

\begin{proof}
Let $V = \mathbb{S}^2 \setminus U$. As $U$ is simply connected and contains at most one critical value of $\pi$, then we can construct a star-like map $T$ with a unique black vertex representing a regular point of $\pi$, and each branched point of $\pi$ in $V$ is associated with a white vertex. Now consider the pull-back of $T$ by $\pi$. Notice that this construction is similar to the one in section \ref{construction}, except here we remove (at most) one of the branched point. Then, by Proposition \ref{propwhitedeg}, each white vertex of a balanced map is incident to a face at most once, so erasing a white vertex and its associated edge in the construction of the balanced map does not change the connectivity of the pull-back map. In conclusion, the pull-back of $T$ by $\pi$ is a connected map, and so, the pull-back of $\mathbb{S}^2 \setminus U$ by $\pi$ is also connected.
\end{proof}

\subsection{Summary}

From a representation $\Gamma$ of a branched covering, we create two different bipartite maps: a $n$-regular skeleton, i.e\@. a globally balanced map whose all the black vertices and faces are of degree $n$, and an increasing skeleton, i.e\@. an increasing, globally balanced map without $1$-valence vertices. More exactly, we define a surjection from the set of representations of a branched covering to the set of $n$-regular skeletons, whose the fiber is generated by the action of the permutation $(1 \cdots n)$ on the labels of the representation. And we define a bijection between the set of representations and the set of increasing skeletons.

Then, by composing the two previous applications, we define a surjection between the set of representations and the set of skeletons, and we prove in Theorem \ref{thbalmap} that a skeleton defines a unique balanced map (here, a balanced map is supposed without $1$-valence vertices). In the following, there is a summary of the relationships between the different sets of maps considered in this section.

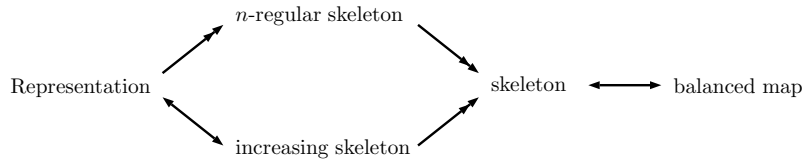
\begin{figure}[ht]
\begin{center}
\psscalebox{0.8 0.8} 
{
\begin{pspicture}(0,-1.255)(13.02,1.255)
\rput[bl](0.0,-0.255){Representation}
\rput[bl](3.7,0.945){$n$-regular skeleton}
\rput[bl](3.7,-1.255){increasing skeleton}
\rput[bl](7.9,-0.155){skeleton}
\rput[bl](10.9,-0.255){balanced map}
\psline[linecolor=black, linewidth=0.04, arrowsize=0.053cm 2.0,arrowlength=1.4,arrowinset=0.0]{<<-}(3.5,0.945)(2.5,0.145)
\psline[linecolor=black, linewidth=0.04, arrowsize=0.053cm 2.0,arrowlength=1.4,arrowinset=0.0]{->>}(6.7,0.945)(7.7,0.145)
\psline[linecolor=black, linewidth=0.04, arrowsize=0.053cm 2.0,arrowlength=1.4,arrowinset=0.0]{<->}(2.5,-0.255)(3.5,-1.055)
\psline[linecolor=black, linewidth=0.04, arrowsize=0.053cm 2.0,arrowlength=1.4,arrowinset=0.0]{->>}(6.7,-1.055)(7.7,-0.255)
\psline[linecolor=black, linewidth=0.04, arrowsize=0.053cm 2.0,arrowlength=1.4,arrowinset=0.0]{<->}(9.5,-0.055)(10.7,-0.055)
\end{pspicture}
}
\end{center}
\end{figure}

\section{From balanced graphs to balanced matrices}

The main result of this section is to give a matrix interpretation of the balanced conditions we defined in the previous sections. First, we come back to the Hurwitz problem, and reprove that the Riemann-Hurwitz condition, mentioned in the introduction of this article, is a necessary condition for the realizability of a given branch datum $\mathcal{D}$ of degree $d$. Then, we consider another invariant: the ramification distribution. In particular, we prove that there exists a surjection from the set of branched coverings of degree $d$ to the set of ramification distribution of degree $d$.

\vspace{2mm}

Let $\mathcal{D} = [\Pi_1, \ldots , \Pi_n]$ be a branch datum of degree $d$, where each $\Pi_i = [k_1 , \ldots , k_{n_i}]$ is a partition of $d$. We recall that the branched weight of $\mathcal{D}$ is equal to
\begin{displaymath}
\nu(\mathcal{D}) = \sum_{i=1}^{n}{\nu(\Pi_i)}, \quad \text{with } \nu(\Pi_i) = \sum_{j=1}^{n_i}{k_j-1}.
\end{displaymath}

\begin{propRHcond}
\label{propRHcond}
(\textbf{Riemann-Hurwitz condition})

If $\mathcal{D}$ is realizable then $\nu(\mathcal{D}) = 2d-2$.
\end{propRHcond}

\begin{proof}
As $\mathcal{D}$ is realizable, there exists a branched covering of the sphere $\pi: \mathbb{S}^2 \rightarrow \mathbb{S}^2$ of degree $d$ such that $\mathcal{D} = \mathcal{D}(\pi)$. Denote by $G$ a skeleton of $\pi$. Then by Theorem \ref{thbalmap}, $G$ is balanced so it has exactly $d$ black vertices and $d$ faces, and moreover the values $k_i > 1$ in the different partitions of $\mathcal{D}$ correspond to the degree of the white vertices of $G$. Now, by Euler characteristic (Theorem \ref{theuler2}), we have:
\begin{displaymath}
(V_G + W_G) + F_G = E_G + 2.
\end{displaymath}
So, as $G$ is globally balanced,
\begin{displaymath}
E_G - W_G = V_G + F_G - 2 = 2 V_G - 2 = 2d-2.
\end{displaymath}
Moreover, by Property \ref{propdegbwvf}, we have the relation
\begin{displaymath}
E_G = \sum_{w \in \mathcal{W}_G}{deg(w)},
\end{displaymath}
and so
\begin{displaymath}
2d-2 = E_G - W_G = \sum_{w \in \mathcal{W}_G}{deg(w)} - W_G = \sum_{w \in \mathcal{W}_G}{\left(deg(w)-1 \right)}.
\end{displaymath}
\end{proof}

Notice that we don't use the local balanced condition of $G$ to prove that the Riemann-Hurwitz condition is a necessary condition. So, as a skeleton of a branched covering is necessarily a globally and locally balanced map, we can  deduce that the Riemann-Hurwitz condition is not a sufficient condition. So now, the natural question is to succeed in translating the local balanced condition on the bipartite map into a condition on the passport in order to find a necessary and sufficient condition for the Hurwitz problem.

Nevertheless, such a relation is not so simple as a passport has no information about the structure of the map $G$. More exactly, a given passport can correspond to several bipartite maps, as a passport has just information about the degree of white vertices. Conversely, a balanced map has no information on the number of critical values of the branched covering, and so a balanced map also corresponds to several passports. That's why we consider the following problem: let $(a_i)$ be a ramification distribution of degree $d$, can we construct a realized passport of degree $d$ using the values $a_i$?

First, we give a matrix interpretation of the balanced condition by using Theorem \ref{thbalmaps}, and Corollary \ref{corbalmaps}.

\vspace{2mm}

In fact, by Corollary \ref{corbalmaps}, if $G$ is a balanced map, then each map equivalent (as graph) to $G$ is also balanced. So, we can restrict our study to graph structure. The interest is that we can associate to a given bipartite graph a white-to-black {\bf incident matrix} $A=(a_{i,j})$ by the following construction: enumerate separately the white and black vertices. Then, the coefficient $a_{i,j}$ is equal to the number of edges having the white vertex $i$ and the black vertex $j$ as tips. Conversely, it's simple to construct a bipartite graph from a matrix containing only non-negative integers by using the same construction. Another interest of this construction is that the degree of a white (resp. black) vertex is just the sum of each coefficient in the corresponding line (resp. column).

Now, the main idea is to construct a ``good'' matrix that corresponds to a given ramification distribution $l$. 

We will denote by $\mathcal{M}_{m,d}(\{0,1\})$ the set of $m\times d$ matrices whose entries are either $0$ or $1$. 

\begin{defassmat}
Let $l=(a_1 , \ldots , a_m)$ be a ramification distribution of degree $d$. We say that a matrix $A \in \mathcal{M}_{m,d}(\{0,1\})$ is a \textbf{matrix representation} of $l$ if for all $i$, the $i$-th line of $A$ contains exactly $a_i$ values $1$.
\end{defassmat}

For example, a matrix representation of the ramification distribution $(2,2,2,2,2,2)$ of degree $4$ is given in Figure \ref{figpass1}.

\begin{figure}[ht]
\centering
\begin{minipage}[r]{.48\linewidth}
\centering
\begin{displaymath}
\begin{pmatrix}
1 & 1 & 0 & 0 \\
0 & 1 & 1 & 0 \\
0 & 0 & 1 & 1 \\
1 & 1 & 0 & 0 \\
0 & 1 & 1 & 0 \\
0 & 0 & 1 & 1
\end{pmatrix}
\end{displaymath}
\end{minipage} \hfill
\begin{minipage}[l]{.48\linewidth}
\centering
\psscalebox{0.7}{
\begin{pspicture}(0,-2.4985578)(1.8013892,2.4985578)
\psdots[linecolor=black, dotsize=0.2](0.9006946,2.4)
\psdots[linecolor=black, dotsize=0.2](0.9006946,0.8)
\psdots[linecolor=black, dotsize=0.2](0.9006946,-0.8)
\psdots[linecolor=black, dotsize=0.2](0.9006946,-2.4)
\psline[linecolor=black, linewidth=0.04](0.9006946,2.4)(0.10069458,1.6)(0.9006946,0.8)(0.10069458,0.0)(0.9006946,-0.8)(0.10069458,-1.6)(0.9006946,-2.4)(1.7006946,-1.6)(0.9006946,-0.8)(1.7006946,0.0)(0.9006946,0.8)(1.7006946,1.6)(0.9006946,2.4)
\psdots[linecolor=black, dotstyle=o, dotsize=0.2](0.10069458,1.6)
\psdots[linecolor=black, dotstyle=o, dotsize=0.2](0.10069458,0.0)
\psdots[linecolor=black, dotstyle=o, dotsize=0.2](0.10069458,-1.6)
\psdots[linecolor=black, dotstyle=o, dotsize=0.2](1.7006946,-1.6)
\psdots[linecolor=black, dotstyle=o, dotsize=0.2](1.7006946,0.0)
\psdots[linecolor=black, dotstyle=o, dotsize=0.2](1.7006946,1.6)
\end{pspicture}
}
\end{minipage}
\caption{Example of a matrix representation of $(2,2,2,2,2,2)$ and its associated graph.}
\label{figpass1}
\end{figure}
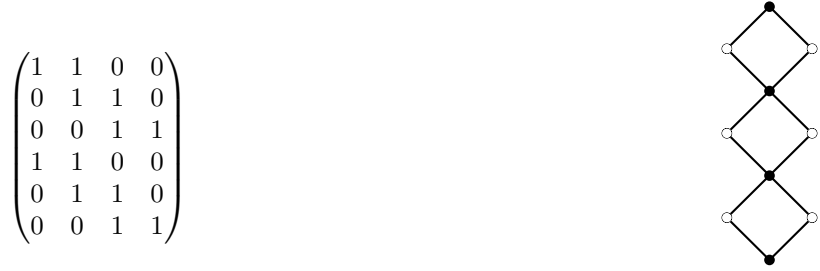

Then, we define a balanced condition for our matrix $A$ as follows:

\begin{defloccond}
We say that a matrix $A \in \mathcal{M}_{m,d}(\{0,1\})$ satisfies the \textbf{balanced condition} if for any choice of $k (\geq 1)$ distinct columns $d_1 , \ldots , d_k$ of $A$, we have
\begin{displaymath}
\sum_{i=1}^{m}\left[ \max\left( 1,\sum_{j=1}^{k}{a_{i,d_j}} \right) -1 \right] \leq 2k-2\quad  \text{and} \quad
\sum_{i=1}^{m}\left[ \sum_{j=1}^{d}{a_{i,j}}  -1 \right] =2d-2.
\end{displaymath}
\end{defloccond}

For example, the previous matrix representation of $(2,2,2,2,2,2)$ satisfies the balanced condition. The reason of the maximum between $1$ and the sum is to avoid the line whose each coefficient $a_{i,d_j} = 0$. Now, using Theorem \ref{thbalmaps}, we prove that this condition for a matrix is equivalent to the balanced conditions for a bipartite map.

\begin{thcond}
\label{thcond}
Let $A \in \mathcal{M}_{m,d}(\{0,1\})$ be the incident matrix of a planar bipartite graph $G$. Then $G$ is balanced if and only if $A$ satisfies the balanced condition.
\end{thcond}

\begin{proof}
Notice first that as the sum of each coefficient in a given line is equal to the degree of the corresponding white vertex, the following condition
\begin{displaymath}
\sum_{i=1}^{m}\left[ \sum_{j=1}^{d}{a_{i,j}} -1 \right] = 2d-2,
\end{displaymath}
is exactly the same as the Riemann-Hurwitz condition. So, using Proposition \ref{propRHcond} if $G$ is globally balanced, then this condition is satisfied. The converse is obviously true.

Now, suppose that $G$ is balanced. We want to prove that $A$ satisfies the balanced condition. Choose arbitrarily $k<d$ columns of $A$, denoted $d_1 , \ldots , d_k$. By definition of $A$, each column corresponds to a black vertex of $G$. So, consider the submap (or the disjoint union of submaps) $H$ of $G$ constructed by erasing each black vertices of $G$ except those corresponding to the columns $d_i$ of $A$, as well as the edge coming from these vertices. Then, by applying Theorem \ref{thbalmaps} in each disjoint component of $H$, we deduce that:
\begin{displaymath}
\sum_{w \in \mathcal{W}_H}{(\deg_H(w)-1)} \leq 2 V_H - 2 = 2k-2.
\end{displaymath}
Moreover, by construction, we have
\begin{displaymath}
\deg_H(w) = \sum_{j=1}^{k}{a_{i,d_j}}, \quad \text{for a fixed } i.
\end{displaymath}
And so,
\begin{displaymath}
\sum_{w \in \mathcal{W}_H}{(\deg_H(w)-1)} = \sum_{i=1}^{m}\left[ \max\left( 1,\sum_{j=1}^{k}{a_{i,d_j}} \right) -1 \right] \leq 2k-2.
\end{displaymath}

\vspace{2mm}

Conversely, suppose that the matrix $A$ satisfies the balanced condition. Left to choose a representation of $G$ on the sphere, we can suppose that $G$ is a planar map. We still prove that $G$ is globally balanced using Proposition \ref{propRHcond}. Now, let's prove that $G$ is also locally balanced. Let $H$ be a submap of $G$, then we can associate a matrix to $H$ as previously. Moreover, this matrix is just a submatrix of $A$, i.e\@. is obtained by erasing some columns and lines of $A$. In particular, we deduce that there exists integers $d_1 , \ldots , d_k$, where $k$ is the number of black vertices of $H$, such that
\begin{displaymath}
\sum_{w \in \mathcal{W}_H}{(\deg_H(w)-1)} \leq \sum_{i=1}^{m}\left[ \max\left( 1,\sum_{j=1}^{k}{a_{i,d_j}} \right) -1 \right].
\end{displaymath}
As $A$ satisfies the balanced condition, we deduce that 
\begin{displaymath}
\sum_{w \in \mathcal{W}_H}{(\deg_H(w)-1)} \leq 2k-2 = 2V_H - 2.
\end{displaymath}
By Theorem \ref{thbalmaps}, we deduce immediately that $G$ is locally balanced.
\end{proof}

Now, we can finally answer our initial question:

\begin{thexists}
\label{thexists}
Let $l=(a_1, \ldots , a_m)$ be a ramification distribution of degree $d$. Then there exists a branched covering of degree $d$ thus the ramification number are given by $l$.
\end{thexists}

\begin{proof}
Consider the matrix $A \in \mathcal{M}_{m,d}(\{0,1\})$ constructed by the following recursive process: first place in the first line of $A$ a value $1$ in each column from $1$ to $a_1$. Now, suppose we still construct the $i$ first lines of $A$ and denote by $b_j$ the value $a_1 + \ldots + a_j - j + 1$, for $j=1, \ldots , n$. We suppose that $b_0=1$. If $b_i<d$, place in the $(i+1)$th line of $A$ a value $1$ in each column from $b_i$ to $b_{i+1}$, by coming back to $1$ if the number of the column exceeds $d$. Else place a value $1$ in each column from $b_i - d + 1$ to $b_{i+1} - d + 1$. 
By doing this process, we construct a matrix representation $A$ of $l$. See for example the previous matrix representation of $(2,2,2,2,2,2)$ in Figure \ref{figpass1}, or the example in Figure \ref{figpass2}. Now we have to prove that this matrix satisfies the balanced condition. Obviously, we still know that
\begin{displaymath}
\sum_{i=1}^{m}\left[ \sum_{j=1}^{d}{a_{i,j}} -1 \right] = \sum_{i=1}^{m}[a_i - 1] = 2d-2,
\end{displaymath}
as $l$ is a ramification distribution of degree $d$. To prove the inequalities, we need to notice two points coming from the construction of $A$. The first point is that $a_{m,d} = 1$, i.e\@. the last value $1$ we place by our process is always in the last column of $A$. Moreover, we reach the last column of $A$ exactly twice. This first point comes directly from the property satisfied by $l$:
\begin{displaymath}
\sum_{i=1}^{m}[a_i - 1] = 2d-2.
\end{displaymath}
In particular, each column of $A$ has necessarily $2$, $3$ or $4$ values $1$. The second point is a consequence of the first one: for any two columns $j$, $k$ of $A$, with $j \neq k$, there exists at most two lines such that $a_{i,j} = a_{i,k} = 1$. In other words, the inner product of the two columns is at most equals to $2$. Then, it's not difficult to prove that $A$ satisfies the balanced condition using these two properties. Moreover, the graph coming from $A$ is clearly planar using Kuratowski's theorem \cite{Kur}. We can also prove directly the graph is planar by constructing it with a similar process as the one proposed to construct the matrix $A$ (see Figures \ref{figpass1} and \ref{figpass2} for some explicit examples).

Finally, using Theorem \ref{thcond}, we prove that there exists a balanced map corresponding to the matrix $A$ (and so to the ramification distribution $l$), and so by Theorem \ref{thbalmap}, we deduce the wanted result.
\end{proof}

\begin{figure}[ht]
\centering
\begin{minipage}[r]{.48\linewidth}
\centering
\begin{displaymath}
\begin{pmatrix}
1 & 1 & 1 & 0 \\
1 & 0 & 1 & 1 \\
0 & 1 & 1 & 0 \\
0 & 0 & 1 & 1
\end{pmatrix}
\end{displaymath}
\end{minipage} \hfill
\begin{minipage}[l]{.48\linewidth}
\centering
\scalebox{0.5}{
\begin{pspicture}(0,-2.86)(3.7,2.86)
\psdots[linecolor=black, dotsize=0.2](1.7,2.76)
\psdots[linecolor=black, dotsize=0.2](1.7,1.16)
\psdots[linecolor=black, dotsize=0.2](1.7,-0.44)
\psdots[linecolor=black, dotsize=0.2](1.7,-2.04)
\psline[linecolor=black, linewidth=0.04](1.7,2.76)(0.1,1.16)(1.7,1.16)(2.5,0.36)(1.7,-0.44)(0.1,1.16)
\psline[linecolor=black, linewidth=0.04](1.7,-0.44)(0.5,-2.04)(1.7,-2.04)(2.5,-1.24)(1.7,-0.44)
\psbezier[linecolor=black, linewidth=0.04](0.5,-2.04)(1.3,-2.84)(2.5,-3.24)(3.3,-2.04)(4.1,-0.84)(3.3,1.96)(1.7,2.76)
\psdots[linecolor=black, dotstyle=o, dotsize=0.2](0.1,1.16)
\psdots[linecolor=black, dotstyle=o, dotsize=0.2](0.5,-2.04)
\psdots[linecolor=black, dotstyle=o, dotsize=0.2](2.5,0.36)
\psdots[linecolor=black, dotstyle=o, dotsize=0.2](2.5,-1.24)
\end{pspicture}
} 
\end{minipage}
\caption{Example of a matrix representation of $(3,3,2,2)$ and its associated graph.}
\label{figpass2}
\end{figure}
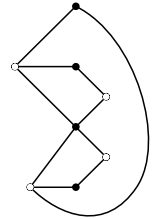

\nocite{DPS}

\vspace{4mm}

\textbf{Acknowledge:} I am grateful to Dylan Thurston and Gilles Schaeffer for many useful conversations. I also thank Tan Lei for introducing the subject, and for her constant help during the preparation of this work.

\bibliographystyle{alpha}
\bibliography{thesebib}

\begin{thebibliography}{Tom14}

\bibitem[Bar01]{Bar}
K.~Baranski.
\newblock On realizability of branched coverings of the sphere.
\newblock {\em Topology and its Applications}, 116:279--291, 2001.

\bibitem[Ber85]{berge}
C.~Berge.
\newblock {\em Graphs}, volume~6.
\newblock North-Holland Mathematical Library, 1985.

\bibitem[DPS14]{DPS}
E.~Duchi, D.~Poulalhon, and G.~Schaeffer.
\newblock Uniform random sampling of simple branched coverings of the sphere by
  itself.
\newblock {\em Proceedings of the Twenty-Fifth Annual ACM-SIAM}, pages
  294--304, 2014.

\bibitem[EKS84]{EKS}
A.~L. Edmonds, R.~S. Kulkarni, and R.~E. Stong.
\newblock Realizability of branched coverings of surfaces.
\newblock {\em Transactions of the American Mathematical Society},
  282:773--790, 1984.

\bibitem[FS09]{FS}
P.~Flajolet and R.~Sedgewick.
\newblock {\em Analytic Combinatorics}.
\newblock Cambridge University Press, 2009.

\bibitem[Ger87]{Ger}
S.M. Gersten.
\newblock On branched covers of the 2-sphere by the 2-sphere.
\newblock {\em Proceedings of the American Mathematical Society}, 101:761--766,
  1987.

\bibitem[GT87]{tucker}
J.L. Gross and T.W. Tucker.
\newblock {\em Topological graph theory}.
\newblock Wiley-Interecience Series in Discrete Mathematics and Optimization,
  1987.

\bibitem[Hal35]{hall}
P.~Hall.
\newblock On {R}epresentatives of {S}ubsets.
\newblock {\em Journal of the London Math. Soc.}, 10:26--30, 1935.

\bibitem[KT15]{Tan}
S.~Koch and L.~Tan.
\newblock On balanced planar graphs, following {W}. {T}hurston.
\newblock preprint, 2015.

\bibitem[Kur30]{Kur}
K.~Kuratowski.
\newblock Sur le probl\`{e}me des courbes gauches en topologie.
\newblock {\em Fundamenta Mathematicae}, 15:271--283, 1930.

\bibitem[Tom14]{Tomth}
J.~Tomasini.
\newblock {\em G\'{e}om\'{e}trie combinatoire des fractions rationnelles}.
\newblock PhD thesis, Universit\'{e} d'Angers, 2014.

\bibitem[Tut84]{tutbook}
W.T. Tutte.
\newblock {\em Graph theory}.
\newblock Encyclopedia of Mathematics and its Applications, 1984.

\bibitem[Zhe06]{Zheng}
H.~Zheng.
\newblock Realizability of branched coverings of ${S}^2$.
\newblock {\em Topology and its Applications}, 153:2124--2134, 2006.

\end{thebibliography}

\end{document}